\documentclass[reqno]{amsart}

 \usepackage{mathtools} % For \coloneqq, \DeclarePairedDelimiter
\usepackage{bm}

\usepackage{amsmath}
\usepackage{amsfonts}
\usepackage{amssymb,amscd,amsthm}
\usepackage{geometry}
\usepackage{mathrsfs}
\usepackage[bookmarksnumbered, colorlinks, linktocpage, plainpages]{hyperref}
\usepackage{nameref}
\usepackage{wasysym,amssymb,eufrak,indentfirst,color}
\usepackage{color}
\usepackage{mathtools}
\usepackage{cite}
\usepackage{appendix}
\usepackage{extarrows}
\usepackage{setspace}

\usepackage[all]{xy}

\usepackage{cases}
\usepackage{graphicx}
\usepackage{ragged2e}

\usepackage{amsmath, amssymb, amsthm}
\usepackage{hyperref}

\newtheorem{theorem}{Theorem}[section]
\newtheorem{lemma}[theorem]{Lemma}
\newtheorem{proposition}[theorem]{Proposition}
\newtheorem{corollary}[theorem]{Corollary}
\theoremstyle{definition}
\newtheorem{definition}[theorem]{Definition}
\theoremstyle{remark}
\newtheorem{remark}[theorem]{Remark}

\numberwithin{equation}{section}
\newcommand{\YY}{\mathbf{Y}}

\begin{document}

\title[Weighted $L^2$ theory for the Euclidean Dirac operator]{Weighted $L^2$ theory for the Euclidean Dirac operator in higher dimensions}

\author{Guangbin Ren}
	\email[G.~Ren]{rengb@ustc.edu.cn}
\address[Guangbin Ren]{Department of Mathematics, University of Science and Technology of China, Hefei 230026, China}

\author{Yuchen Zhang$^\ast$}
\email[Y. Zhang]{yuchen95@amss.ac.cn}
\address[Yuchen Zhang]{Institute of Mathematics, Academy of Mathematics and Systems Sciences, Chinese Academy of Sciences, Beijing 100190, China}

\begin{abstract}
We study weighted $L^{2}$ solvability for the Euclidean Dirac operator in dimensions $n\ge 3$. We prove that, on the exterior domain $\mathbb{R}^{n}\setminus\overline{B(0,1)}$ with logarithmic weight $\varphi=n\log|x|$, no higher-dimensional analogue of the two-dimensional H\"ormander estimate can be controlled solely by $\Delta\varphi$; we then establish weighted solvability for the weights $|x|^{m}$ with $m\neq 0$, for the quadratic weight $x_{1}^{2}$, and for sufficiently small anisotropic perturbations of the Gaussian weight, with sharp constant $1/4$ in the Gaussian case. The obstruction arises because, in dimensions $n\ge 3$, the classical weighted identity is coercive only under a structural relation between $\Delta\varphi$ and $|\nabla\varphi|^{2}$, a condition that excludes the Gaussian weight and many polynomial weights. The method is based on a weighted identity for the conjugated unknown $U:=ue^{-\varphi/2}$, together with suitable scalar and Clifford-valued multipliers; this identity yields the required coercive estimates and also gives weighted $L^{2}$ solvability for the Poisson equation through the factorization $\Delta=-D^{2}$.
\end{abstract}



\maketitle

\section{Introduction}\label{section:1}

H\"{o}rmander's $L^{2}$ method for the $\bar{\partial}$-equation is a standard tool in several complex variables and related areas of analysis
\cite{Hormander65,Hormander73,Hormander03History,Demailly-smallbook,DemaillyBook,McNealVarolin15,Ohsawa18}.
In the Euclidean Dirac setting, \cite{JiZhu17} proved in dimension two a weighted $L^{2}$ existence theorem under the sole assumption that the weight is subharmonic. We recall this result for reference.

\begin{theorem}\label{thm:1.1}
Let $\Omega \subset \mathbb{R}^{2}$ be a domain and let $\varphi \in C^{2}(\Omega,\mathbb{R})$ be subharmonic. Suppose that
\[
\int_{\Omega} \frac{|f|^{2}}{\Delta\varphi} e^{-\varphi}\,dV < \infty
\]
for some $f \in L^{2}_{\varphi}(\Omega,\mathbb{R}_{2})$. Then there exists $u \in L^{2}_{\varphi}(\Omega,\mathbb{R}_{2})$ such that
\[
Du=f \quad \text{in }\Omega,
\]
and
\[
\|u\|_{\varphi}^{2} \le \int_{\Omega} \frac{|f|^{2}}{\Delta\varphi} e^{-\varphi}\,dV.
\]
\end{theorem}

In dimensions $n\ge3$, the classical weighted identity yields coercivity only under an additional relation between $\Delta\varphi$ and $|\nabla\varphi|^{2}$. This condition excludes, in particular, the Gaussian weight and many polynomial weights. The following theorem shows that this restriction reflects an actual obstruction.

\begin{theorem}[Obstruction in higher dimensions]\label{thm:1.2}
Let $n \ge 3$, $\Omega = \mathbb{R}^{n} \setminus \overline{B(0,1)}$, and $\varphi = n\log|x|$. For each $m \in \mathbb{N}_{+}$ define
\[
u_{m} = |x|^{-1/m}, \qquad f_{m} = D u_{m} = -\frac{1}{m}|x|^{-1/m-2}x.
\]
Then $Du_{m}=f_{m}$, $u_{m}$ is the solution of minimal $L^{2}_{\varphi}$-norm, and
\[
\lim_{m\to+\infty}
\frac{\|u_{m}\|_{\varphi}^{2}}
{\displaystyle \int_{\Omega}\frac{|f_{m}|^{2}}{\Delta\varphi}e^{-\varphi}\,dV}
=+\infty.
\]
In particular, for this pair $(\Omega,\varphi)$ the naive higher-dimensional H\"{o}rmander-type existence statement fails: there is no finite constant $C=C(\Omega,\varphi)$ such that every $f\in L^{2}_{\varphi}(\Omega,\mathbb{R}_{n})$ with
\[
\int_{\Omega} \frac{|f|^{2}}{\Delta\varphi} e^{-\varphi}\,dV < \infty
\]
admits a solution $u$ of $Du=f$ satisfying
\[
\|u\|_{\varphi}^{2} \le C \int_{\Omega} \frac{|f|^{2}}{\Delta\varphi} e^{-\varphi}\,dV.
\]
\end{theorem}

Theorem~\ref{thm:1.2} shows that, in higher dimensions, weighted solvability cannot in general be based on $\Delta\varphi$ alone. The problem is therefore to recover coercivity by a different argument. The next theorem gives such results for several classes of weights.

\begin{theorem}[Weighted $L^2$ solvability]\label{thm:1.3}
Let $n \ge 2$ and let $\Omega \subset \mathbb{R}^{n}$ be a domain. Then the following statements hold.
\begin{itemize}
 \item[(1)] Let $m \in \mathbb{R}\setminus\{0\}$, let $\varphi = |x|^{m}$, and assume $\Omega \subset \mathbb{R}^{n}\setminus\{0\}$. For every
 $f \in L^{2}_{\varphi}(\Omega,\mathbb{R}_{n})$ satisfying
 \[
 \int_{\Omega} \frac{|f|^{2}}{|x|^{m-2}} e^{-\varphi}\,dV < \infty,
 \]
 there exists $u \in L^{2}_{\varphi}(\Omega,\mathbb{R}_{n})$ such that $Du = f$ and
 \[
 \|u\|_{\varphi}^{2} \le
 \int_{\Omega} \frac{|f|^{2}}{m^{2}|x|^{m-2}} e^{-\varphi}\,dV.
 \]
 In particular, for $\varphi=|x|^{2}$ one has
 \begin{equation}\label{eq:sharp-estimate-gauss-weight}
 \|u\|_{\varphi}^{2} \le \frac{1}{4}\,\|f\|_{\varphi}^{2},
 \end{equation}
 and, if $\Omega=\mathbb{R}^{n}\setminus\{0\}$, the constant $1/4$ is sharp.

 \item[(2)] Let $\varphi = x_{1}^{2}$. For every $f \in L^{2}_{\varphi}(\Omega,\mathbb{R}_{n})$ there exists a solution $u$ of $Du = f$ such that
 \[
 \|u\|_{\varphi}^{2} \le \frac{1}{2}\,\|f\|_{\varphi}^{2}.
 \]

 \item[(3)] Let $\varphi = \sum_{i=1}^{n} a_{i}x_{i}^{2}$ with $|a_{i}-1|<\varepsilon$ for some sufficiently small $\varepsilon>0$, and assume that $\Omega$ is exterior to $B(0,1)$. Then for every $f \in L^{2}_{\varphi}(\Omega,\mathbb{R}_{n})$ there exists a solution $u$ of $Du = f$ such that
 \[
 \|u\|_{\varphi}^{2} \le \frac{1}{3}\,\|f\|_{\varphi}^{2}.
 \]
\end{itemize}
\end{theorem}

Part~(1) includes the Gaussian estimate with sharp constant $1/4$. Part~(2) treats the quadratic weight $x_{1}^{2}$, and part~(3) gives stability under small anisotropic quadratic perturbations on exterior domains. Via the factorization $\Delta=-D^{2}$, these estimates also yield weighted solvability for the Poisson equation.

The following corollary records the corresponding consequence for the Poisson equation.

\begin{corollary}\label{cor:1.4}
Let $\varphi=|x|^{2}$ and let $\Omega \subset \mathbb{R}^{n}\setminus\{0\}$ be a domain. For every $f \in L^{2}_{\varphi}(\Omega,\mathbb{R}_{n})$ there exists $u \in L^{2}_{\varphi}(\Omega,\mathbb{R}_{n})$ such that
\[
\Delta u=f \quad \text{in }\Omega
\]
and
\[
\|u\|_{\varphi}^{2} \le \frac{1}{16}\,\|f\|_{\varphi}^{2}.
\]
If $f$ is real-valued, then $u$ can be chosen real-valued with the same estimate.
\end{corollary}

The proof relies on the interaction between the weighted adjoint of the Dirac operator and Clifford multiplication, whose non-commutativity makes the order of the factors essential. We therefore work with the conjugated unknown $U:=ue^{-\varphi/2}$ and introduce auxiliary scalar and Clifford-valued multipliers. This leads to a weighted identity from which the radial, single-quadratic, and perturbed-Gaussian estimates are derived.

The paper is organized as follows. Section~\ref{section:2} recalls the basic Clifford algebra background. Section~\ref{section:3} sets up the weighted $L^{2}$ framework for the Dirac operator and its adjoint. Section~\ref{section:4} reviews the classical Ji--Zhu theory and the structural condition underlying it. Section~\ref{section:5} proves the obstruction in higher dimensions. Section~\ref{section:6} develops the weighted identity and the corresponding coercive estimates. Section~\ref{section:7} derives the existence theorems, and Section~\ref{section:8} contains concluding remarks.

\section{Clifford algebras and Dirac operators}\label{section:2}

This section recalls the standard algebraic and analytic background for real Clifford algebras and the Euclidean Dirac operator; see \cite{BrackxDelangheSommen82,DelangheSommenSoucek92,GilbertMurray91}.

Let $\mathbb R^n$ be equipped with its standard Euclidean structure and orthonormal basis
$\{e_1,\dots,e_n\}$. The \emph{real Clifford algebra} $\mathbb R_n$ is the associative algebra over
$\mathbb R$ generated by $e_1,\dots,e_n$ and $1$ subject to the relations
\begin{equation}\label{eq:Clifford-relations}
 e_j^2 = -1,\qquad j=1,\dots,n,
\end{equation}
and
\begin{equation}\label{eq:Clifford-anticommutation}
 e_j e_k + e_k e_j = 0,\qquad 1 \le j < k \le n.
\end{equation}
Thus $e_j e_k = - e_k e_j$ whenever $j\ne k$, and the algebra is generated by $1$ and the $e_j$
modulo these relations.

Notable examples of Clifford algebras include $\mathbb{R}_1\cong \mathbb{C}$ and $\mathbb{R}_2\cong \mathbb{H}$.
As a real vector space, $\mathbb R_n$ has dimension $2^n$. A convenient basis is given by
the ordered monomials
\[
 e_A := e_{j_1} e_{j_2} \cdots e_{j_\ell},
\]
where $A = \{j_1<\dots<j_\ell\}\subset\{1,\dots,n\}$ and by convention
$e_\emptyset := 1$. Every element $f\in\mathbb R_n$ can be written uniquely in the form
\begin{equation}\label{eq:Clifford-expansion}
 f = \sum_A f_A e_A,
\end{equation}
with real coefficients $f_A\in\mathbb R$. The scalar part of $f$ is the coefficient of $e_\emptyset$,
which we denote by
\[
 \operatorname{Re}(f) := f_\emptyset.
\]

We write the vector variable as
\[x:=\sum_{j=1}^n x_j e_j.\]

There is a standard conjugation on $\mathbb R_n$, defined as the unique algebra
anti-automorphism
\[\overline{\phantom{f}} : \mathbb R_n\to\mathbb R_n\]
 such that
\[
 \overline{1} = 1,\qquad \overline{e_j} = - e_j\quad (j=1,\dots,n),
\]
and
\[
 \overline{fg} = \overline{g}\,\overline{f},\qquad f,g\in\mathbb R_n.
\]
In particular, for a basis element $e_A$ with $|A|=\ell$ one checks that
\[
 \overline{e_A} = (-1)^{\ell(\ell+1)/2} e_A.
\]

We equip $\mathbb R_n$ with the $\mathbb R$-valued inner product
\begin{equation}\label{eq:Clifford-inner-product}
 \langle f,g\rangle := \operatorname{Re}(f\overline{g}),\qquad f,g\in\mathbb R_n.
\end{equation}
Using the expansion~\eqref{eq:Clifford-expansion} and the orthogonality of the basis
$\{e_A\}_A$ with respect to~\eqref{eq:Clifford-relations}--\eqref{eq:Clifford-anticommutation}, one
obtains
\[
 \langle f,g\rangle = \sum_A f_A g_A.
\]
Thus the associated norm
\begin{equation}\label{eq:Clifford-norm}
 |f|^2 := \langle f,f\rangle = \sum_A |f_A|^2
\end{equation}
agrees with the Euclidean norm on $\mathbb R^{2^n}$ under the identification
$f\longleftrightarrow (f_A)_A$.

As an algebra, $\mathbb R_n$ is not normed in the multiplicative sense: in general one does not
have $|fg| = |f|\,|g|$. The next lemma records a universal bound and a special case in which
multiplicativity does hold.

\begin{lemma}\label{lem:Clifford-norm}
For all $f,g\in\mathbb R_n$ one has
\[
 |fg|^2\le 2^n |f|^2\,|g|^2.
\]
Moreover, if $f$ belongs to the \emph{paravector} subspace
\[
 \mathrm{span}_\mathbb R\{1,e_1,\dots,e_n\}\subset \mathbb R_n,
\]
then
\[
 |fg| = |gf| = |f|\,|g|.
\]
\end{lemma}

\begin{proof}
Let $f = \sum_A f_A e_A$ and $g = \sum_B g_B e_B$. Then
\[
 fg = \sum_{A,B} f_A g_B e_A e_B = \sum_C h_C e_C.
\]
Fix a multi-index $C$. For each multi-index $A$ there is a unique multi-index
$B = A\triangle C$ such that
\[
 e_A e_B = \sigma(A,C)e_C,
 \qquad \sigma(A,C)\in\{\pm1\},
\]
where $A\triangle C$ denotes the symmetric difference. Hence
\[
 h_C = \sum_A \sigma(A,C) f_A g_{A\triangle C}.
\]
Therefore,
\[
 |h_C|
 \le \sum_A |f_A|\,|g_{A\triangle C}|
 \le \Bigl(\sum_A |f_A|^2\Bigr)^{1/2}
    \Bigl(\sum_A |g_{A\triangle C}|^2\Bigr)^{1/2}
 = |f|\,|g|,
\]
because $A\mapsto A\triangle C$ is a permutation of the set of multi-indices.
Summing over the $2^n$ possible values of $C$, we obtain
\[
 |fg|^2 = \sum_C |h_C|^2 \le 2^n |f|^2 |g|^2,
\]
which proves the first assertion.

For the second statement, assume
\[
 f = f_0 + \sum_{j=1}^n f_j e_j
\]
with real coefficients $f_0,f_1,\dots,f_n$. A direct computation gives
\[
 f\overline{f} = \overline{f}\,f
 = f_0^2 + \sum_{j=1}^n f_j^2 = |f|^2 \in \mathbb R.
\]
We also use the elementary identity
\[
 \operatorname{Re}(ab)=\operatorname{Re}(ba),
 \qquad a,b\in\mathbb R_n,
\]
which follows by expanding $a$ and $b$ in the basis $\{e_A\}_A$: only the diagonal terms
$A=B$ contribute to the scalar part. Then, for arbitrary $g\in\mathbb R_n$,
\[
 |fg|^2
 = \operatorname{Re}\bigl((fg)\overline{fg}\bigr)
 = \operatorname{Re}\bigl(fg\overline{g}\,\overline{f}\bigr)
 = \operatorname{Re}\bigl(g\overline{g}\,\overline{f}f\bigr)
 = |f|^2\operatorname{Re}(g\overline{g})
 = |f|^2|g|^2.
\]
Hence $|fg| = |f|\,|g|$. The equality $|gf| = |f|\,|g|$ is obtained in the same way.
\end{proof}

\section{The Dirac operator and its $L^2$ theory}\label{section:3}

This section introduces the Dirac operator on a domain $\Omega\subset\mathbb R^n$ and places it in a weighted $L^2$ framework adapted to the Clifford algebra structure. This will be the analytic setting for the estimates and existence results proved later.

\subsection{Definitions and basic properties}

Let $\Omega\subset\mathbb R^n$ be a domain. We define the (left) Dirac operator as follows.

\begin{definition}[Dirac operator]\label{def:Dirac}
The Dirac operator on $\Omega$ is defined by
\[
 D : C^1(\Omega,\mathbb R_n) \longrightarrow C^0(\Omega,\mathbb R_n),
 \qquad
 D u = \sum_{j=1}^n e_j\, \partial_j u.
\]
\end{definition}

To formulate weighted $L^2$ estimates, we introduce the following Hilbert space.

\begin{definition}[Weighted $L^2$ space]\label{def:L2-weighted}
Let $\varphi\in C(\Omega,\mathbb R)$ be a real-valued weight function.
The weighted space $L^2_\varphi(\Omega,\mathbb R_n)$ consists of all measurable functions
$u:\Omega\to\mathbb R_n$ such that
\[
 \|u\|_\varphi^2
 := \int_\Omega |u|^2 e^{-\varphi}\,dV < \infty.
\]
The associated inner product is
\begin{equation}\label{eq:weighted-inner}
 \langle u,v\rangle_\varphi
 := \int_\Omega \operatorname{Re}\bigl( u \,\overline{v}\bigr)\,e^{-\varphi}\,dV
  = \int_\Omega \sum_A u_A\,v_A\,e^{-\varphi}\,dV,
\end{equation}
where $u = \sum_A u_A e_A$ and $v = \sum_A v_A e_A$ are the coefficient expansions in the
basis $\{e_A\}$. When $\varphi=0$, we write $\langle\cdot,\cdot\rangle_0$ and $\|\cdot\|_0$
for the corresponding unweighted inner product and norm.
\end{definition}

The next lemma summarizes the basic compatibility properties of
$\langle\cdot,\cdot\rangle_\varphi$ with respect to multiplication by scalar functions and by
Clifford generators.

\begin{lemma}\label{lem:weighted-inner-basic}
Let $f,g\in L^2_\varphi(\Omega,\mathbb R_n)$.
\begin{enumerate}
\item If $h\in C(\Omega,\mathbb R)$, then
\[
 \langle f,hg\rangle_\varphi = \langle hf,g\rangle_\varphi
\]
whenever the integrals are finite.
\item For each $j=1,\dots,n,$ one has
\[
 \langle f, e_j g\rangle_\varphi = - \langle e_j f, g\rangle_\varphi.
\]
\item Let
\[
 \Psi = \psi_0 + \sum_{j=1}^n \psi_j e_j,
 \qquad \psi_0,\psi_j\in C^\infty(\Omega,\mathbb R).
\]
Then for all $f,g\in C_c^\infty(\Omega,\mathbb R_n)$,
\[
 \langle f,\Psi g\rangle_\varphi = \langle \overline{\Psi}\, f, g\rangle_\varphi,
 \qquad
 \overline{\Psi}=\psi_0-\sum_{j=1}^n \psi_j e_j.
\]
\end{enumerate}
\end{lemma}

\begin{proof}
Assertion (1) follows directly from the definition~\eqref{eq:weighted-inner} and the fact that
$h$ is real-valued. For (2), the anticommutation relations~\eqref{eq:Clifford-anticommutation}
imply $\overline{e_j}=-e_j$. Using also the elementary identity
$\operatorname{Re}(ab)=\operatorname{Re}(ba)$ for $a,b\in\mathbb R_n$, we obtain
\[
 \operatorname{Re}(f\,\overline{e_j g})
 = \operatorname{Re}(f\,\overline{g}\,\overline{e_j})
 = -\operatorname{Re}(f\,\overline{g} e_j)
 = -\operatorname{Re}(e_j f\,\overline{g})
 = \operatorname{Re}\bigl((-e_j f)\,\overline{g}\bigr).
\]
After integration this gives
\[
 \langle f, e_j g\rangle_\varphi = -\langle e_j f, g\rangle_\varphi.
\]
Finally, if
\[
 \Psi = \psi_0 + \sum_{j=1}^n \psi_j e_j,
\]
then by (1) and (2),
\[
 \langle f,\Psi g\rangle_\varphi
 = \langle \psi_0 f,g\rangle_\varphi
  + \sum_{j=1}^n \langle -\psi_j e_j f,g\rangle_\varphi
 = \langle \overline{\Psi}\,f,g\rangle_\varphi.
\]
\end{proof}

\subsection{The maximal operator defined by $D$ and its adjoint}

We next realize $D$ as an unbounded operator on the Hilbert space $L^2_\varphi(\Omega,\mathbb R_n)$. This formulation will be used in the functional-analytic arguments below.

\subsubsection*{Formal adjoint}

 The formal adjoint of $D$ with respect to
$\langle\cdot,\cdot\rangle_\varphi$ is defined as follows.

\begin{definition}[Formal adjoint]\label{def:formal-adjoint}
The formal adjoint $\delta_\varphi$ of $D$ is the differential operator characterized by
\begin{equation}\label{eq:formal-adjoint-def}
 \langle \delta_\varphi u, v\rangle_\varphi
 = \langle u, D v\rangle_\varphi,
 \qquad u,v\in C_c^\infty(\Omega,\mathbb R_n).
\end{equation}
\end{definition}

The next proposition gives an explicit expression for $\delta_\varphi$.

\begin{proposition}\label{prop:formal-adjoint}
Assume $\varphi\in C^1(\Omega,\mathbb R)$. Then the formal adjoint of $D$ with respect to
$\langle\cdot,\cdot\rangle_\varphi$ is given by
\begin{equation}\label{eq:formal-adjoint-formula}
 \delta_\varphi u =  e^{\varphi} D\bigl(u e^{-\varphi}\bigr)
          =  D u - (D\varphi)\,u.
\end{equation}
Here $D\varphi := \sum_{j=1}^n e_j \partial_j \varphi$ is viewed as a Clifford-valued
multiplication operator.
\end{proposition}

\begin{proof}
We first consider the unweighted case $\varphi\equiv 0$. Let
$u,v\in C_c^\infty(\Omega,\mathbb R_n)$. Since
\[
 Dv=\sum_{j=1}^n e_j\partial_j v,
\]
Lemma~\ref{lem:weighted-inner-basic}(2) gives
\[
 \langle u,Dv\rangle_0
 = \sum_{j=1}^n \langle u,e_j\partial_j v\rangle_0
 = -\sum_{j=1}^n \langle e_j u,\partial_j v\rangle_0.
\]
Because $u$ and $v$ are compactly supported and $e_j$ is constant, integration by parts
componentwise yields
\[
 -\langle e_j u,\partial_j v\rangle_0
 = \int_\Omega \operatorname{Re}\!\bigl(\partial_j(e_j u)\,\overline{v}\bigr)\,dV
 = \langle e_j\partial_j u, v\rangle_0.
\]
Summing over $j$ we obtain
\[
 \langle u,Dv\rangle_0 = \langle Du,v\rangle_0,
\]
so $\delta_0=D$.

Now let $\varphi\in C^1(\Omega,\mathbb R)$ be arbitrary and set
\[
 w:=u e^{-\varphi}.
\]
Since $u\in C_c^\infty(\Omega,\mathbb R_n)$ and $\varphi\in C^1(\Omega,\mathbb R)$, we have
$w\in C_c^1(\Omega,\mathbb R_n)$. Using that $e^{-\varphi}$ is real-valued and applying the same
componentwise integration-by-parts argument as above to $w$ and $v$, we get
\[
 \langle u,Dv\rangle_\varphi
 = \int_\Omega \operatorname{Re}\bigl(u\,\overline{Dv}\bigr)e^{-\varphi}\,dV
 = \langle w,Dv\rangle_0
 = \langle Dw,v\rangle_0
 = \left\langle e^{\varphi}Dw,v\right\rangle_\varphi.
\]
Therefore \eqref{eq:formal-adjoint-def} holds with
\[
 \delta_\varphi u = e^{\varphi}D(u e^{-\varphi}).
\]

Finally, the Leibniz rule gives
\[
 D(u e^{-\varphi})
 = \sum_{j=1}^n e_j\partial_j(u e^{-\varphi})
 = (Du)e^{-\varphi} + \sum_{j=1}^n e_j u\,\partial_j(e^{-\varphi}).
\]
Since $\partial_j(e^{-\varphi})=-(\partial_j\varphi)e^{-\varphi}$ is scalar-valued, it commutes with
Clifford multiplication, and therefore
\[
 D(u e^{-\varphi})
 = (Du)e^{-\varphi} - \sum_{j=1}^n e_j(\partial_j\varphi)u\,e^{-\varphi}
 = \bigl(Du-(D\varphi)u\bigr)e^{-\varphi}.
\]
Multiplying by $e^{\varphi}$ yields
\[
 \delta_\varphi u = Du-(D\varphi)u.
\]
In particular, when $\varphi\equiv 0$, this reduces to $\delta_0=D$.
\end{proof}
\subsubsection*{The maximal operator defined by $D$ and its Hilbert adjoint}

We now introduce the maximal operator defined by $D$ in $L^2_\varphi(\Omega,\mathbb R_n)$ and
relate its Hilbert space adjoint to the formal adjoint $\delta_\varphi$.
The construction is completely analogous to the general theory of first-order systems of partial differential equations; see, for example, \cite{Hormander55,Hormander73}.

\begin{definition}[Maximal operator]\label{def:maximal-Dirac}
The \emph{maximal} differential operator defined by $D$ is the unbounded operator
\[
 D : L^2_\varphi(\Omega,\mathbb R_n) \longrightarrow L^2_\varphi(\Omega,\mathbb R_n),
\]
with domain
\[
 \operatorname{Dom}(D)
 := \bigl\{ u\in L^2_\varphi(\Omega,\mathbb R_n) : D u \in L^2_\varphi(\Omega,\mathbb R_n)
   \text{ in the sense of distributions} \bigr\}.
\]
For $u\in\operatorname{Dom}(D)$, the distribution $Du$ is defined by
\[
 \langle Du, v\rangle_0 = \langle u, Dv\rangle_0,\qquad \text{for all }
 v\in C_c^\infty(\Omega,\mathbb R_n),
\]
where $\langle\cdot,\cdot\rangle_0$ denotes the unweighted $L^2$ inner product, and $\|v\|_0^2:=\langle v,v\rangle_0$.
\end{definition}

It is standard that $D$ with this domain is densely defined and closed in
$L^2_\varphi(\Omega,\mathbb R_n)$. Its Hilbert adjoint, denoted by $D^\ast_\varphi$, is defined as follows.
\begin{definition}[Hilbert adjoint of maximal operator]
  An element $v$ belongs to $\operatorname{Dom}(D_\varphi^\ast)$ if there exists $f\in L^2_\varphi(\Omega,\mathbb{R}_n)$ such that
  \[\langle Du,v\rangle_\varphi=\langle u,f\rangle_\varphi \qquad \text{ for all } u\in \operatorname{Dom}(D).\]
For such $v$, we define $D^\ast_\varphi v=f$.
\end{definition}

The next lemma identifies the Hilbert adjoint with the formal adjoint on sufficiently regular functions and records the boundary trace condition encoded in the domain of $D_\varphi^*$.

\begin{lemma}\label{lem:adjoint-and-boundary}
Let $\Omega\subset\mathbb R^n$ be a bounded domain with $C^1$ boundary and
$\varphi\in C^1(\overline{\Omega},\mathbb R)$. If
\[
 u\in C^1(\overline{\Omega},\mathbb R_n)\cap\operatorname{Dom}(D_\varphi^*),
\]
then $D_\varphi^* u = \delta_\varphi u$ in $L^2_\varphi(\Omega,\mathbb R_n)$. Moreover,
\[
 u|_{\partial\Omega} = 0.
\]
\end{lemma}

\begin{proof}
We first identify $D_\varphi^*$ with $\delta_\varphi$.

Let $v\in C_c^\infty(\Omega,\mathbb R_n)$. Since $v\in \operatorname{Dom}(D)$, the definition of $D_\varphi^*$ gives
\[
 \langle Dv, u\rangle_\varphi = \langle v, D_\varphi^* u\rangle_\varphi.
\]
Because $\langle\cdot,\cdot\rangle_\varphi$ is symmetric, we also have
\[
 \langle Dv, u\rangle_\varphi = \langle u, Dv\rangle_\varphi.
\]
Now Proposition~\ref{prop:formal-adjoint} gives
\[
 \delta_\varphi u = e^{\varphi} D\bigl(u e^{-\varphi}\bigr) = Du - (D\varphi)u.
\]
Since $u\in C^1(\overline{\Omega},\mathbb R_n)$ and $v\in C_c^\infty(\Omega,\mathbb R_n)$, the same
componentwise integration-by-parts computation as in the proof of Proposition~\ref{prop:formal-adjoint}
(with no boundary term because $v$ has compact support) yields
\[
 \langle u, Dv\rangle_\varphi = \langle \delta_\varphi u, v\rangle_\varphi.
\]
Using symmetry once more,
\[
 \langle v, D_\varphi^* u\rangle_\varphi
 = \langle Dv, u\rangle_\varphi
 = \langle u, Dv\rangle_\varphi
 = \langle \delta_\varphi u, v\rangle_\varphi
 = \langle v, \delta_\varphi u\rangle_\varphi.
\]
Hence
\[
 \langle v, D_\varphi^* u - \delta_\varphi u\rangle_\varphi = 0
 \qquad \text{for all } v\in C_c^\infty(\Omega,\mathbb R_n).
\]
Since $u\in C^1(\overline{\Omega},\mathbb R_n)$ and $\varphi\in C^1(\overline{\Omega},\mathbb R)$, we have
$\delta_\varphi u\in L^2_\varphi(\Omega,\mathbb R_n)$. Therefore, by density of
$C_c^\infty(\Omega,\mathbb R_n)$ in $L^2_\varphi(\Omega,\mathbb R_n)$, it follows that
$D_\varphi^* u = \delta_\varphi u$ in $L^2_\varphi(\Omega,\mathbb R_n)$.

\medskip

We next show that the boundary trace vanishes. Let $v\in C^1(\overline{\Omega},\mathbb R_n)\cap\operatorname{Dom}(D)$. A standard divergence theorem
computation, applied componentwise in the coefficient representation of Clifford-valued
functions and using Proposition~\ref{prop:formal-adjoint}, gives
\begin{equation}\label{eq:boundary-identity}
 \langle Dv, u\rangle_\varphi
 = \operatorname{Re}\int_{\partial\Omega} \nu\,v\,\overline{u}\,e^{-\varphi}\,dS
  + \langle v, \delta_\varphi u\rangle_\varphi,
\end{equation}
where $\nu = \sum_{j=1}^n \nu_j e_j$ denotes the outward unit normal to $\partial\Omega$ and
$dS$ is the surface measure.

On the other hand, since $u\in\operatorname{Dom}(D_\varphi^*)$ and
$D_\varphi^* u = \delta_\varphi u$, we also have
\[
 \langle Dv, u\rangle_\varphi = \langle v, D_\varphi^* u\rangle_\varphi
 = \langle v, \delta_\varphi u\rangle_\varphi.
\]
Comparing this with~\eqref{eq:boundary-identity} we obtain
\[
 \operatorname{Re}\int_{\partial\Omega} \nu\,v\,\overline{u}\,e^{-\varphi}\,dS = 0
\]
for all $v\in C^1(\overline{\Omega},\mathbb R_n)\cap\operatorname{Dom}(D)$.

Suppose instead that the boundary trace of $u$ does not vanish identically. Since
$u\in C^1(\overline{\Omega},\mathbb R_n)$, there exist a point $y\in\partial\Omega$ with
$u(y)\neq 0$ and an open neighborhood $W$ of $y$ in $\mathbb R^n$ such that
$u\neq 0$ on $W\cap\partial\Omega$. Since $\partial\Omega$ is $C^1$, the unit normal field
extends to a $C^1$ Clifford-valued vector field $\widetilde{\nu}$ on $W$ with
$\widetilde{\nu}=\nu$ on $W\cap\partial\Omega$. Choose a cut-off function
$\rho\in C_c^\infty(W)$, $0\le\rho\le 1$, with $\rho\equiv 1$ near $y$, and set
\[
 v := -\rho\,\widetilde{\nu}\,u \quad\text{in }\Omega.
\]
Then $v\in C^1(\overline{\Omega},\mathbb R_n)\cap\operatorname{Dom}(D)$. On
$W\cap\partial\Omega$ we have
\[
 \operatorname{Re}(\nu\,v\,\overline{u})
 = \operatorname{Re}\bigl(-\rho\,\nu\,\widetilde{\nu}\,u\,\overline{u}\bigr)
 = -\rho\,\operatorname{Re}(\nu^{2}u\overline{u})
 = \rho\,\operatorname{Re}(u\overline{u})
 = \rho\,|u|^{2},
\]
because $\widetilde{\nu}=\nu$ on $\partial\Omega$ and $\nu^{2}=-|\nu|^{2}=-1$.
Therefore
\[
 \operatorname{Re}\int_{\partial\Omega} \nu\,v\,\overline{u}\,e^{-\varphi}\,dS
 = \int_{\partial\Omega\cap W} \rho\,|u|^2 e^{-\varphi}\,dS > 0,
\]
since $\rho\equiv1$ near $y$ and $u\neq 0$ on a neighborhood of $y$ in $\partial\Omega$.
This contradicts the vanishing of the boundary integral established above. Hence $u|_{\partial\Omega}=0$, as claimed.
\end{proof}

\subsection{Functional analysis preliminaries for the $L^2$ existence theorem}

We record the functional-analytic lemma used in the $L^{2}$ existence argument.
\begin{lemma}\label{lem:functional-analysis}
Let $f\in L^2_\varphi(\Omega,\mathbb{R}_n)$. Suppose that for all $u\in C_c^\infty(\Omega,\mathbb{R}_n)$,
\begin{equation}\label{eq:goal estimate}
  |\langle u,f\rangle_\varphi|^2 \le C\,\|D_\varphi^* u\|_\varphi^2
\end{equation}
for some constant $C\ge0$. Then there exists $g\in L^2_\varphi(\Omega,\mathbb{R}_n)$ such that
\[
  Dg=f
\]
in the sense of distributions and
\[
  \|g\|_\varphi^2 \le C.
\]
\end{lemma}
\begin{proof}
We construct a linear functional $L_f$ by
\[
  L_f(D_\varphi^*u):=\langle u,f\rangle_\varphi,
  \qquad u\in C_c^\infty(\Omega,\mathbb{R}_n).
\]
The solution $g$ will then be obtained from the Riesz representation theorem.

We first define $L_f$ on the image of $D_\varphi^*$, namely on the subspace
\[
  E=\{D_\varphi^*u: u\in C_c^\infty(\Omega,\mathbb{R}_n)\}.
\]
For $v=D_\varphi^*u\in E$, we set
\[
  L_f(v):=\langle u,f\rangle_\varphi.
\]

This is well defined. Indeed, if $D_\varphi^*(u_1-u_2)=0$, then \eqref{eq:goal estimate} gives
\[
  |\langle u_1-u_2,f\rangle_\varphi|^2
  \le C\,\|D_\varphi^*(u_1-u_2)\|_\varphi^2 = 0,
\]
so $\langle u_1,f\rangle_\varphi=\langle u_2,f\rangle_\varphi$.

By the Hahn--Banach theorem, $L_f$ extends to a continuous linear functional on $L^2_\varphi(\Omega,\mathbb{R}_n)$, still denoted by $L_f$, with
\[
  |L_f(v)| \le \sqrt{C}\,\|v\|_\varphi,
  \qquad v\in L^2_\varphi(\Omega,\mathbb{R}_n).
\]
By the Riesz representation theorem, there exists $g\in L^2_\varphi(\Omega,\mathbb{R}_n)$ such that
\[
  L_f(v)=\langle v,g\rangle_\varphi,
  \qquad v\in L^2_\varphi(\Omega,\mathbb{R}_n),
\]
and $\|g\|_\varphi\le\sqrt{C}$.

To verify that $Dg=f$ distributionally, let $U\in C_c^\infty(\Omega,\mathbb R_n)$ be arbitrary and set $u:=Ue^{\varphi}$. Then $u\in C_c^\infty(\Omega,\mathbb R_n)$ and, taking $v=D_\varphi^*u$, we obtain
\[
  \langle U,f\rangle_0
  =\langle u,f\rangle_\varphi
  =L_f(D_\varphi^*u)
  =\langle D_\varphi^*u,g\rangle_\varphi
  =\langle DU,g\rangle_0.
\]
Hence $Dg=f$ in the sense of distributions. Since $\|g\|_\varphi\le\sqrt{C}$, we also have $\|g\|_\varphi^2\le C$.
\end{proof}

The following lemma is the corresponding $L^2$ existence criterion.

\begin{lemma}\label{lem:L2 existence}
Let $\Omega\subset \mathbb{R}^n$ be a domain and let $A$ be a non-negative function over $\Omega$. Suppose that
\begin{equation}\label{eq:lem:L2 existence}
\|D^\ast_\varphi u\|^2_\varphi\geq \int_\Omega A |u|^2e^{-\varphi}dV
\end{equation}
holds for every $u\in C_c^\infty(\Omega,\mathbb{R}_n)$. Then for all $f\in L^2_\varphi(\Omega,\mathbb{R}_n)$ satisfying
\[\int_\Omega \frac{|f|^2}{A}e^{-\varphi}dV<\infty,\]
there exists $g\in L^2_\varphi(\Omega,\mathbb{R}_n)$ such that $Dg=f$ with the estimate
\[\|g\|^2_\varphi\leq \int_\Omega \frac{|f|^2}{A}e^{-\varphi}dV. \]
\end{lemma}
\begin{proof}
Let $f\in L^2_\varphi(\Omega,\mathbb{R}_n)$ satisfying that
\[\int_\Omega \frac{|f|^2}{A}e^{-\varphi}dV<\infty.\]
For any $u\in C_c^\infty(\Omega,\mathbb{R}_n)$, by the Cauchy–Schwarz inequality and the hypothesis \eqref{eq:lem:L2 existence}, we have
\begin{align*}
|\langle f,u\rangle_\varphi|^2
&= \left|\left\langle \frac{f}{\sqrt{A}},\sqrt{A}u\right\rangle_\varphi\right|^2 \\
&\leq \left( \int_\Omega\frac{|f|^2}{A}e^{-\varphi}dV \right)
    \left( \int_\Omega A|u|^2e^{-\varphi}dV \right) \\
&\leq \|D_\varphi^* u\|^2_\varphi
   \int_\Omega \frac{|f|^2}{A}e^{-\varphi}dV.
\end{align*}
Taking
\[C=\int_\Omega \frac{|f|^2}{A}e^{-\varphi}dV,\]
we may apply Lemma \ref{lem:functional-analysis} to conclude that there exists $g\in L^2_\varphi(\Omega,\mathbb{R}_n)$ such that $Dg=f$ with the estimate
\[\|g\|^2_\varphi\leq C= \int_\Omega \frac{|f|^2}{A}e^{-\varphi}dV. \]
This completes the proof.
\end{proof}

\begin{remark}
The density lemma plays a crucial role in the classical $L^{2}$–method for the $\bar\partial$–complex.
Unlike H\"{o}rmander's original approach~\cite{Hormander65}, the $L^2$ theory for the Dirac equation developed here does not require such a density result. This simplification arises because we only consider the single equation $Dg=f$, whereas H\"{o}rmander's theory concerns the system
\[\bar{\partial}g=f,\quad \bar{\partial}f=0,\]
which is inherently more complicated.
\end{remark}

\section{Classical weighted $L^{2}$ theory for the Dirac operator}\label{section:4}

We recall the classical weighted $L^{2}$ identity for the Dirac operator in the Euclidean setting. The structural condition
\begin{equation*}
 \Delta \varphi - \Bigl(1 - \frac{2}{n}\Bigr)\Bigl(1 + \frac{1}{\varepsilon}\Bigr) |\nabla\varphi|^{2} \ge 0,
\end{equation*}
will serve as a reference point in the later sections. For the Gaussian weight $\varphi = |x|^{2}$ on $\mathbb{R}^{n}$ with $n \ge 3$, this condition fails.

\subsection{The fundamental identity}

The following proposition records the weighted identity relating $D^{*}_{\varphi}$ to the derivatives of the weight $\varphi$. It is the Dirac analogue of the classical H\"{o}rmander identity for $\bar\partial$ and reformulates \cite[Proposition~2.5]{JiZhu17} in the present setting.

\begin{proposition}[Weighted $L^{2}$ identity]\label{prop:fundamental-identity}
Let $n \ge 2$ and let $\Omega \subset \mathbb{R}^{n}$ be a domain.
Let $\varphi \in C^{2}(\Omega,\mathbb{R})$ and let constants $\kappa,k \in \mathbb{R}$
satisfy
\begin{equation}\label{eq:kappa-k-relation}
 \kappa > \frac{n-2}{n},
 \qquad
 k = \frac{(n-2)\kappa}{n\kappa - (n-2)} > \frac{n-2}{n}.
\end{equation}
Then, for every $u \in C_{c}^{\infty}(\Omega,\mathbb{R}_{n})$, the identity
\begin{align}
 (1+k)\,\|D^{*}_{\varphi} u\|^{2}_{\varphi}
 &= \int_{\Omega} \bigl(\Delta \varphi - \kappa |\nabla \varphi|^{2}\bigr) |u|^{2} e^{-\varphi}\,dV
   \label{eq:weighted-L2-identity}\\
 &\quad + \int_{\Omega}
    \left|
     \sqrt{k - \frac{n-2}{n}}\, Du
     - \sqrt{\kappa + k}\, D\varphi \cdot u
    \right|^{2} e^{-\varphi}\,dV \nonumber\\
 &\quad + 2 \int_{\Omega}
   \left(
    \sum_{j,A} |\partial_{j} u_{A}|^{2}
    - \frac{1}{n} |Du|^{2}
   \right) e^{-\varphi}\,dV
   \nonumber
\end{align}
holds.
\end{proposition}

\begin{proof}
Set
\[
 A:=\|D_\varphi^*u\|_\varphi^2,\qquad
 B:=\|Du\|_\varphi^2,\qquad
 N:=\int_\Omega |\nabla\varphi|^2|u|^2e^{-\varphi}\,dV.
\]
A reformulation of the computation leading to the basic Bochner identity is
\begin{equation}\label{eq:basic-Bochner}
 A+B
 =2\int_\Omega \sum_{j,A}|\partial_j u_A|^2e^{-\varphi}\,dV
  +\int_\Omega \Delta\varphi\,|u|^2e^{-\varphi}\,dV.
\end{equation}
Hence
\begin{equation}\label{eq:weighted-identity-step1}
 A
 = \int_\Omega \Delta\varphi\,|u|^2e^{-\varphi}\,dV
  +2\int_\Omega\left(\sum_{j,A}|\partial_j u_A|^2-\frac1n|Du|^2\right)e^{-\varphi}\,dV
  -\frac{n-2}{n}\,B.
\end{equation}

On the other hand, since $D_\varphi^*u=Du-D\varphi\cdot u$, we have
\begin{equation}\label{eq:weighted-identity-step2}
 A
 = B + N - 2\langle Du, D\varphi\cdot u\rangle_\varphi.
\end{equation}
Multiplying \eqref{eq:weighted-identity-step2} by $k$ and adding the result to
\eqref{eq:weighted-identity-step1}, we obtain
\begin{align*}
 (1+k)A
  &= \int_\Omega \Delta\varphi\,|u|^2e^{-\varphi}\,dV
   +2\int_\Omega\left(\sum_{j,A}|\partial_j u_A|^2-\frac1n|Du|^2\right)e^{-\varphi}\,dV \\
  &\quad +\left(k-\frac{n-2}{n}\right)B + kN
   -2k\langle Du,D\varphi\cdot u\rangle_\varphi.
\end{align*}
Now add and subtract $\kappa N$:
\begin{align*}
 (1+k)A
  &= \int_\Omega (\Delta\varphi-\kappa|\nabla\varphi|^2)|u|^2e^{-\varphi}\,dV \\
  &\quad +2\int_\Omega\left(\sum_{j,A}|\partial_j u_A|^2-\frac1n|Du|^2\right)e^{-\varphi}\,dV \\
  &\quad +\left(k-\frac{n-2}{n}\right)B + (\kappa+k)N
   -2k\langle Du,D\varphi\cdot u\rangle_\varphi.
\end{align*}
The relation \eqref{eq:kappa-k-relation} is equivalent to
\[
 \left(k-\frac{n-2}{n}\right)(\kappa+k)=k^2.
\]
Therefore the last line is exactly
\[
 \int_\Omega
    \left|
     \sqrt{k - \frac{n-2}{n}}\, Du
     - \sqrt{\kappa + k}\, D\varphi \cdot u
    \right|^{2} e^{-\varphi}\,dV,
\]
and \eqref{eq:weighted-L2-identity} follows.
\end{proof}

\begin{remark}
The last integral in \eqref{eq:weighted-L2-identity} is nonnegative:
\[
 \sum_{j,A} |\partial_{j} u_{A}|^{2} - \frac{1}{n} |Du|^{2} \ge 0,
\]
which is the standard pointwise inequality for the Dirac operator and
expresses the fact that the trace-free part of the gradient dominates the
trace part. This will be used to derive a priori estimates by discarding
positive terms.
\end{remark}

\subsection{A priori estimate and the existence theorem}

Proposition~\ref{prop:fundamental-identity} yields the following a priori estimate for $D^{*}_{\varphi}$. Combined with Lemma~\ref{lem:functional-analysis}, it recovers the classical higher-dimensional existence theorem under the structural condition $\Delta \varphi \ge \kappa |\nabla\varphi|^{2}$, corresponding to \cite[Proposition~2.8]{JiZhu17} in the Euclidean setting.

\begin{proposition}[A priori estimate]\label{prop:apriori}
\begin{enumerate}
\item Let $\Omega \subset \mathbb{R}^{2}$ be a domain and let $\varphi \in C^{2}(\Omega,\mathbb{R})$ be a subharmonic function. Then the following inequality
  \begin{equation*}
    \|D^{*}_{\varphi} u\|^{2}_{\varphi}
 \;\ge\; \int_{\Omega} \Delta \varphi\,|u|^{2} e^{-\varphi}\,dV
  \end{equation*}
 holds for all $u\in C^\infty_c(\Omega,\mathbb{R}_2)$.
\item Let $n>2$, $\Omega \subset \mathbb{R}^{n}$ be a domain,
and let $\varphi \in C^{2}(\Omega,\mathbb{R})$. Assume that
\[
 \kappa > \frac{n-2}{n}.
\]
Then there exists a constant $C_{\kappa} > 0$, depending only on $n$ and $\kappa$,
such that for all $u \in C_c^\infty(\Omega,\mathbb{R}_n)$,
\begin{equation}\label{eq:apriori-estimate}
 \|D^{*}_{\varphi} u\|^{2}_{\varphi}
 \;\ge\; C_{\kappa}
      \int_{\Omega}
       \bigl(\Delta \varphi - \kappa |\nabla\varphi|^{2}\bigr)
       |u|^{2} e^{-\varphi}\,dV.
\end{equation}
\end{enumerate}
\end{proposition}

\begin{proof}
We first prove the assertion $(1)$. Let $\Omega\subset\mathbb{R}^2$ be a domain and let $\varphi$ be a subharmonic function over $\Omega$. Recall from equation \eqref{eq:basic-Bochner} that for any $u\in C_c^\infty(\Omega,\mathbb{R}_n)$,
\begin{align*}
  \|D^{*}_{\varphi} u\|^{2}_{\varphi}&=
   -\|Du\|^{2}_{\varphi}+
  2 \int_{\Omega} \sum_{j,A} |\partial_{j} u_{A}|^{2} e^{-\varphi}\,dV
   + \int_{\Omega}\Delta\varphi \,|u|^{2} e^{-\varphi}\,dV\\
   &=\int_{\Omega} |(e_1\partial_{1}-e_2\partial_2) u|^{2} e^{-\varphi}\,dV
   + \int_{\Omega}\Delta\varphi \,|u|^{2} e^{-\varphi}\,dV\geq \int_{\Omega}\Delta\varphi \,|u|^{2} e^{-\varphi}\,dV.
\end{align*}
This proves assertion $(1)$. The assertion $(2)$ follows from Proposition \ref{prop:fundamental-identity} and the fact that
\[
 \sum_{j,A} |\partial_{j} u_{A}|^{2} - \frac{1}{n} |Du|^{2} \ge 0.
\]
\end{proof}

Combining Proposition~\ref{prop:apriori} and Lemma~\ref{lem:L2 existence}, we obtain
the following classical existence theorem.

\begin{theorem}[Classical $L^{2}$-existence theorem]\label{thm:classical-L2}
\begin{enumerate}
\item Let $\Omega \subset \mathbb{R}^{2}$ be a domain and let $\varphi \in C^{2}(\Omega,\mathbb{R})$ be a subharmonic function. Then, for every $f \in L^{2}_{\varphi}(\Omega,\mathbb{R}_2)$ such that
\begin{equation*}
 \int_{\Omega}
  \frac{|f|^{2}}{\Delta\varphi}
  e^{-\varphi}\,dV < \infty,
\end{equation*}
there exists a solution $u \in L^{2}_{\varphi}(\Omega,\mathbb{R}_2)$ to the
Dirac equation
\[
 Du = f \quad\text{in }\Omega
\]
satisfying the estimate
\begin{equation*}
 \|u\|^{2}_{\varphi}
  \;\le\;
      \int_{\Omega}
       \frac{|f|^{2}}{\Delta\varphi}
       e^{-\varphi}\,dV.
\end{equation*}
\item Let $n>2$, $\Omega \subset \mathbb{R}^{n}$ be a domain and
$\varphi \in C^{2}(\Omega,\mathbb{R})$ a weight function satisfying
\begin{equation}\label{eq:structural-kappa}
 \Delta \varphi \;\ge\; \kappa |\nabla\varphi|^{2}
 \quad\text{for some constant }\kappa > \frac{n-2}{n}.
\end{equation}
Then, for every $f \in L^{2}_{\varphi}(\Omega,\mathbb{R}_{n})$ such that
\begin{equation}\label{eq:compatibility}
 \int_{\Omega}
  \frac{|f|^{2}}{\Delta\varphi - \kappa |\nabla\varphi|^{2}}
  e^{-\varphi}\,dV < \infty,
\end{equation}
there exists a solution $u \in L^{2}_{\varphi}(\Omega,\mathbb{R}_{n})$ to the
Dirac equation
\[
 Du = f \quad\text{in }\Omega
\]
satisfying the estimate
\begin{equation}\label{eq:classical-estimate}
 \|u\|^{2}_{\varphi}
  \;\le\; \frac{1}{C_{\kappa}}
      \int_{\Omega}
       \frac{|f|^{2}}{\Delta\varphi - \kappa |\nabla\varphi|^{2}}
       e^{-\varphi}\,dV,
\end{equation}
where $C_{\kappa}>0$ is the constant from Proposition~\ref{prop:apriori},
depending only on $n$ and $\kappa$.
\end{enumerate}
\end{theorem}

\begin{proof}
This theorem is a direct corollary of Lemma \ref{lem:L2 existence} and Proposition \ref{prop:apriori}.
\end{proof}

\begin{remark}
Theorem~\ref{thm:classical-L2} is the Euclidean counterpart of
\cite[Proposition~2.8]{JiZhu17}, where the same structural condition
\eqref{eq:structural-kappa} is imposed on the weight on a spin manifold.
Sections~\ref{section:6} and \ref{section:7} show that this condition is too restrictive in higher
dimension: the Gaussian weight $\varphi=|x|^{2}$, standard in analysis and probability, does not satisfy \eqref{eq:structural-kappa} on $\mathbb{R}^{n}\setminus\{0\}$ when $n \ge 3$.
\end{remark}

\section{An obstruction in higher dimensions}\label{section:5}

Section~\ref{section:4} yields in dimension $n=2$ a H\"{o}rmander-type existence theorem under the sole assumption $\Delta\varphi\ge0$. In this section we show that such a statement does not extend to dimensions $n\ge3$. More precisely, we exhibit a subharmonic weight on an exterior domain for which there is no finite constant $C=C(\Omega,\varphi)$ such that every $f\in L^{2}_{\varphi}(\Omega,\mathbb{R}_{n})$ with
\[
 \int_{\Omega}\frac{|f|^{2}}{\Delta\varphi}\,e^{-\varphi}\,dV<\infty
\]
admits a solution $u$ of $Du=f$ satisfying
\begin{equation}\label{eq:naive-estimate}
 \|u\|^{2}_{\varphi}
  \le C \int_{\Omega}\frac{|f|^{2}}{\Delta\varphi}\,e^{-\varphi}\,dV.
\end{equation}
The counterexample identifies an obstruction to any higher-dimensional theory based only on $\Delta\varphi$.

Within the classical Bochner framework, the stronger condition
\begin{equation}\label{eq:structure-kappa}
 \Delta\varphi \;\ge\; \kappa\,|\nabla\varphi|^{2},
 \qquad \kappa>\frac{n-2}{n},
\end{equation}
should therefore be regarded as a coercivity hypothesis rather than a technical convenience.

\subsection{The counterexample}\label{subsec:counterexample}

We work on the exterior domain
\[
 \Omega = \mathbb{R}^{n}\setminus \overline{B(0,1)}
\]
with the radial weight
\begin{equation}\label{eq:phi-log}
 \varphi = n\log|x|, \qquad x\in\Omega.
\end{equation}
A direct computation shows that $\varphi$ is subharmonic on $\Omega$. Indeed,
\[
 \Delta\varphi = n\,\Delta(\log|x|)
  = \frac{n(n-2)}{|x|^{2}} \ge 0 \quad\text{for }|x|>1.
\]
Thus $\varphi$ satisfies the natural subharmonic condition $\Delta\varphi\ge0$, and
\eqref{eq:naive-estimate} would be the direct higher-dimensional counterpart of the
two-dimensional theorem if it were valid.

For each $m\in\mathbb{N}_{+}$ we introduce a radial function
\[
 u_{m} = |x|^{-1/m}, \qquad x\in\Omega,
\]
and set
\[
 f_{m} = Du_{m}.
\]
A straightforward computation using the fact that $u_{m}$ is radial and that
$D = \sum_{j=1}^{n} e_{j}\partial_{j}$ yields
\begin{equation}\label{eq:def-um-fm}
 f_{m}
  = Du_{m}
  = -\,\frac{1}{m}\,|x|^{-1/m-2}\,x,
 \qquad |x|>1.
\end{equation}
In particular, $u_{m}$ and $f_{m}$ are smooth on $\Omega$ and satisfy the Dirac equation
\[
 Du_{m} = f_{m} \quad\text{in }\Omega.
\]

The following theorem establishes the obstruction.

\begin{theorem}[Obstruction]\label{thm:sharp-obstruction}
Let $n\ge3$, $\Omega=\mathbb{R}^{n}\setminus \overline{B(0,1)}$, and $\varphi=n\log|x|$.
Let $u_{m}$ and $f_{m}$ be defined as in \eqref{eq:def-um-fm}. Then:
\begin{enumerate}
 \item For each $m$, $u_{m}$ is the unique solution of
    \[
     Du = f_{m}\quad\text{in }\Omega
    \]
    with minimal $L^{2}_{\varphi}$–norm in $L^{2}_{\varphi}(\Omega,\mathbb{R}_{n})$.
 \item The ratio between the squared norm of the solution and the weighted norm of the
    data diverges:
    \begin{equation}\label{eq:divergence-ratio}
     \lim_{m\to+\infty}
     \frac{\|u_{m}\|^{2}_{\varphi}}{\displaystyle
         \int_{\Omega}\frac{|f_{m}|^{2}}{\Delta\varphi}\,e^{-\varphi}\,dV}
     \;=\; +\infty.
    \end{equation}
\end{enumerate}
In particular, for this pair $(\Omega,\varphi)$ the naive higher-dimensional H\"{o}rmander-type existence statement fails: there is no finite constant $C>0$ such that every $f\in L^{2}_{\varphi}(\Omega,\mathbb{R}_{n})$ with
\[
 \int_{\Omega}\frac{|f|^{2}}{\Delta\varphi}\,e^{-\varphi}\,dV < \infty
\]
admits a solution $u\in L^{2}_{\varphi}(\Omega,\mathbb{R}_{n})$ of $Du=f$ satisfying
\[
 \|u\|^{2}_{\varphi}
  \le C \int_{\Omega}\frac{|f|^{2}}{\Delta\varphi}\,e^{-\varphi}\,dV.
\]
\end{theorem}

\subsection{Proof of Theorem~\ref{thm:sharp-obstruction}}\label{subsec:proof-thm52}

The proof has two parts. We first compute the ratio in \eqref{eq:divergence-ratio}; we then verify that $u_{m}$ is the minimal-norm solution among all $L^{2}_{\varphi}$–solutions of $Du=f_{m}$.

\subsubsection*{Computation of the norms}

Since $\varphi=n\log|x|$, we have
\[
 e^{-\varphi} = |x|^{-n} \quad\text{for }|x|>1.
\]
The problem is rotationally invariant, so we set
\[
 r = |x|,\qquad \omega = \frac{x}{|x|}\in S^{n-1},
\]
and use spherical coordinates. Let $\sigma_{n-1}$ denote the surface area of the unit sphere
$S^{n-1}\subset\mathbb{R}^{n}$.

First, we compute the norm of $u_{m}$:
\begin{align*}
 \|u_{m}\|^{2}_{\varphi}
  &= \int_{\Omega} |u_{m}|^{2} e^{-\varphi}\,dV
  = \int_{|x|>1} |x|^{-2/m}\,|x|^{-n}\,dV \\
  &= \int_{1}^{\infty}\int_{S^{n-1}} r^{-2/m}\,r^{-n}\,r^{n-1}\,dS(\omega)\,dr \\
  &= \sigma_{n-1}\int_{1}^{\infty} r^{-2/m-1}\,dr
  = \sigma_{n-1}\,\frac{m}{2}.
\end{align*}
Hence
\begin{equation}\label{eq:norm-um}
 \|u_{m}\|^{2}_{\varphi} = \frac{m}{2}\,\sigma_{n-1}.
\end{equation}

Next, we compute the weighted norm of $f_{m}$ with the factor $\Delta\varphi$.
From \eqref{eq:def-um-fm} we obtain
\[
 |f_{m}|^{2}
  = \frac{1}{m^{2}}\,|x|^{-2/m-4}\,|x|^{2}
  = \frac{1}{m^{2}}\,|x|^{-2/m-2},
\]
and we already know $\Delta\varphi = n(n-2)/|x|^{2}$ and $e^{-\varphi}=|x|^{-n}$. Thus
\begin{align*}
 \frac{|f_{m}|^{2}}{\Delta\varphi}\,e^{-\varphi}
  &= \frac{\displaystyle\frac{1}{m^{2}}|x|^{-2/m-2}}{\displaystyle\frac{n(n-2)}{|x|^{2}}}\,
   |x|^{-n} \\
  &= \frac{1}{m^{2} n(n-2)}\,|x|^{-2/m}\,|x|^{-n}.
\end{align*}
Integrating in spherical coordinates yields
\begin{align}
 \int_{\Omega}\frac{|f_{m}|^{2}}{\Delta\varphi}\,e^{-\varphi}\,dV
  &= \frac{1}{m^{2} n(n-2)}
   \int_{1}^{\infty}\int_{S^{n-1}} r^{-2/m}\,r^{-n}\,r^{n-1}\,dS(\omega)\,dr \nonumber\\
  &= \frac{\sigma_{n-1}}{m^{2} n(n-2)}
   \int_{1}^{\infty} r^{-2/m-1}\,dr
  = \frac{\sigma_{n-1}}{2 m n(n-2)}. \label{eq:norm-fm}
\end{align}

Combining \eqref{eq:norm-um} and \eqref{eq:norm-fm} we find
\begin{equation}\label{eq:ratio-explicit}
 \frac{\|u_{m}\|^{2}_{\varphi}}{\displaystyle
     \int_{\Omega}\frac{|f_{m}|^{2}}{\Delta\varphi}\,e^{-\varphi}\,dV}
  = \frac{\displaystyle\frac{m}{2}\sigma_{n-1}}{\displaystyle\frac{\sigma_{n-1}}{2 m n(n-2)}}
  = m^{2} n(n-2).
\end{equation}
In particular,
\[
 \frac{\|u_{m}\|^{2}_{\varphi}}{\displaystyle
     \int_{\Omega}\frac{|f_{m}|^{2}}{\Delta\varphi}\,e^{-\varphi}\,dV}
 \sim C(n)\,m^{2}\quad\text{as }m\to+\infty,
\]
with $C(n)=n(n-2)>0$, which proves the divergence \eqref{eq:divergence-ratio}.

\subsubsection*{Minimality of $u_{m}$}

To complete the proof of Theorem~\ref{thm:sharp-obstruction}, it remains to show that for each
$m$ the function $u_{m}$ is the unique solution of $Du=f_{m}$ in $L^{2}_{\varphi}(\Omega,\mathbb{R}_{n})$
with minimal norm. Equivalently, we must prove that $u_{m}$ is orthogonal in $L^{2}_{\varphi}$ to the
null space of $D$, i.e.\ to the space
\[
 \mathcal{M}_{\varphi}(\Omega)
  = \{h\in L^{2}_{\varphi}(\Omega,\mathbb{R}_{n}) : Dh=0\}
\]
of monogenic functions in the weighted space. In other words, we must verify that
\begin{equation}\label{eq:orthogonality-um}
 \langle u_{m},h\rangle_{\varphi} = 0
 \quad\text{for all }h\in\mathcal{M}_{\varphi}(\Omega),
\end{equation}
where
\[
 \langle u,h\rangle_{\varphi}
  := \int_{\Omega} \operatorname{Re}(u\overline{h})\,e^{-\varphi}\,dV
\]
is the weighted inner product introduced in Section~\ref{section:3}.

The key input is the Laurent expansion for monogenic functions on exterior domains in Clifford
analysis, equivalently the decomposition into inner and outer spherical monogenics; see, for example,
\cite{DelangheSommenSoucek92,GilbertMurray91} and the references therein. For the corresponding
Kelvin transform (inversion) for monogenic functions, see \cite[Example~3.7]{Delanghe01History}.
With our convention
\[
 D=\sum_{j=1}^{n} e_j\,\partial_j,
\]
the relevant Kelvin transform for \emph{left monogenic} functions is
\[
 (K_{\mathrm{mon}} g)(x)
 := \frac{x}{|x|^{n}}\, g\!\left(\frac{x}{|x|^{2}}\right),
 \qquad x\neq 0.
\]
The following lemma records the corresponding Dirac identity for the Kelvin transform.

\begin{lemma}\label{lem:Kelvin-left-monogenic}
Let $U\subset\mathbb R^n\setminus\{0\}$ be open and let
\[
 I(U):=\left\{x\in\mathbb R^n\setminus\{0\}: \frac{x}{|x|^2}\in U\right\}
\]
be its image under the Euclidean inversion $x\mapsto x/|x|^2$. If $g\in C^1(U,\mathbb R_n)$, then
for every $x\in I(U)$,
\begin{equation}\label{eq:Kelvin-Dirac-identity}
 D_x\!\left(\frac{x}{|x|^n} g\!\left(\frac{x}{|x|^2}\right)\right)
 = -\frac{x}{|x|^{n+2}}\,(Dg)\!\left(\frac{x}{|x|^2}\right).
\end{equation}
In particular, if $Dg=0$ on $U$, then $K_{\mathrm{mon}}g$ is left monogenic on $I(U)$.
\end{lemma}

\begin{proof}
Write $r:=|x|$, $y:=x/r^2$, and $J(x):=x r^{-n}$. First,
\[
 \begin{aligned}
 DJ
 &= \sum_{i=1}^n e_i\,\partial_i(xr^{-n})\\
 &= \sum_{i=1}^n e_i\bigl(e_i r^{-n} + x\,\partial_i(r^{-n})\bigr)\\
 &= -n r^{-n} - n r^{-n-2}\sum_{i=1}^n x_i e_i x\\
 &= -n r^{-n} - n r^{-n-2}x^2
 = 0,
 \end{aligned}
\]
since $x^2=-r^2$. Next, let
\[
 g_j(y):=(\partial_{y_j}g)(y),\qquad j=1,\dots,n.
\]
Because
\[
 \partial_{x_i}y_j = \delta_{ij} r^{-2} - 2x_i x_j r^{-4},
\]
the product rule and the chain rule give
\begin{align*}
 D_x\bigl(J(x)g(y)\bigr)
  &= (DJ)g(y) + \sum_{i=1}^n e_i J(x)\,\partial_{x_i}\bigl(g(y)\bigr) \\
  &= \sum_{i,j=1}^n e_i J(x) g_j(y)\,\partial_{x_i}y_j \\
  &= r^{-n-2}\sum_{j=1}^n e_j x\,g_j(y)
   -2r^{-n-4}\sum_{j=1}^n \Bigl(\sum_{i=1}^n x_i e_i x\Bigr) x_j g_j(y) \\
  &= r^{-n-2}\sum_{j=1}^n \bigl(e_jx+2x_j\bigr)g_j(y).
\end{align*}
Using the Clifford relation $e_jx+xe_j=-2x_j$, we obtain $e_jx+2x_j=-xe_j$, and therefore
\[
 D_x\bigl(J(x)g(y)\bigr)
 = -r^{-n-2}x\sum_{j=1}^n e_j g_j(y)
 = -\frac{x}{|x|^{n+2}}\,(Dg)(y),
\]
which is exactly \eqref{eq:Kelvin-Dirac-identity}.
\end{proof}

\begin{remark}
Another common convention uses the Clifford inverse
\[
 x^{-1}=\frac{\overline{x}}{|x|^2}=-\frac{x}{|x|^2},
\]
since $\overline{x}=-x$ for vectors. The corresponding Kelvin transform differs from
$K_{\mathrm{mon}}$ only by the reflection $x\mapsto -x$, and reflection preserves left monogenicity
because $D[g(-x)]=-(Dg)(-x)$. We use the normalization above because if $M_j$ is homogeneous of
degree $j$, then
\[
 M_j\!\left(\frac{x}{|x|^2}\right)=|x|^{-2j}M_j(x),
\]
and hence
\begin{equation}\label{eq:Kelvin-homogeneous}
 (K_{\mathrm{mon}}M_j)(x)=\frac{x\,M_j(x)}{|x|^{n+2j}}.
\end{equation}
Thus the outer homogeneous monogenic terms are precisely Kelvin transforms of homogeneous
monogenic polynomials.
\end{remark}

By elliptic regularity,
every $h\in\mathcal{M}_{\varphi}(\Omega)$ is smooth on $\Omega$, and on $\Omega$ it admits the expansion
\begin{equation}\label{eq:Laurent-monogenic}
 h(x)=\sum_{j=0}^{\infty} P_{j}(x)+\sum_{j=0}^{\infty}
 \frac{x\,M_{j}(x)}{|x|^{n+2j}},
\end{equation}
with local uniform convergence on $\Omega$, where each $P_j$ and each $M_j$ is a homogeneous left
monogenic polynomial of degree $j$.

Because $e^{-\varphi}=|x|^{-n}$ and $h\in L^{2}_{\varphi}(\Omega,\mathbb{R}_{n})$, the polynomial part
must vanish identically: a nonzero homogeneous monogenic polynomial $P_{j}$ of degree $j$ satisfies
$|P_{j}(r\omega)|\asymp r^{j}$ on some spherical cap, and therefore
\[
 \int_{|x|>R} |P_{j}(x)|^{2} e^{-\varphi}\,dV
 \gtrsim \int_{R}^{\infty} r^{2j-1}\,dr = +\infty.
\]
Hence the expansion of $h$ on $\Omega$ reduces to the outer monogenic part,
\begin{equation}\label{eq:Kelvin-monogenic}
 h(x)=\sum_{k=1}^{\infty} Q_{k}(x),
 \qquad
 Q_k(x):=\frac{x\,M_{k-1}(x)}{|x|^{n+2k-2}},
 \qquad |x|>1,
\end{equation}
again with local uniform convergence.

For later use, set
\[
 \widetilde Q_k(x):=x\,M_{k-1}(x).
\]
Then each $\widetilde Q_k$ is a homogeneous harmonic polynomial of degree $k$, because
\[
 \Delta \widetilde Q_k
 = \Delta\bigl(xM_{k-1}\bigr)
 = 2DM_{k-1}+x\Delta M_{k-1}=0,
\]
where we used that monogenic polynomials are harmonic. Writing $x=r\omega$ with $r>1$ and
$\omega\in S^{n-1}$, homogeneity gives
\[
 Q_k(r\omega)=r^{-(n+k-2)}\,\widetilde Q_k(\omega).
\]
For each fixed $r>1$ the series \eqref{eq:Kelvin-monogenic} converges uniformly on $rS^{n-1}$, so we may
integrate termwise over the sphere. Each component of $\widetilde Q_k(\omega)$ is a spherical harmonic of
positive degree $k$ and is therefore orthogonal on $S^{n-1}$ to constants; see, for example, \cite{ArmitageGardiner01}. In particular,
\begin{equation}\label{eq:spherical-mean-zero}
 \int_{S^{n-1}} \operatorname{Re}\!\bigl(h(r\omega)\bigr)\,dS(\omega)=0,
 \qquad r>1.
\end{equation}

Now $u_{m}(r)=r^{-1/m}$ is radial, and by the Cauchy--Schwarz inequality
\[
 \int_{\Omega} |u_{m}|\,|h|\,e^{-\varphi}\,dV
 \le \|u_{m}\|_{\varphi}\,\|h\|_{\varphi}<\infty.
\]
Therefore Fubini's theorem applies. Using $e^{-\varphi}=r^{-n}$ and
$dV=r^{n-1}\,dr\,dS(\omega)$, we obtain
\[
 \langle u_{m},h\rangle_{\varphi}
 = \int_{1}^{\infty} r^{-1/m-1}
   \left(\int_{S^{n-1}} \operatorname{Re}\!\bigl(h(r\omega)\bigr)\,dS(\omega)\right)dr
 =0
\]
by \eqref{eq:spherical-mean-zero}. This proves \eqref{eq:orthogonality-um}.

Since $u_{m}$ is orthogonal to the null space of $D$ in $L^{2}_{\varphi}$, it is the unique solution of
$Du=f_{m}$ of minimal norm, by the usual Hilbert space projection argument (see, e.g.,
\cite{Hormander73,DemaillyBook}). This completes the proof of
Theorem~\ref{thm:sharp-obstruction}.

\subsection{Interpretation and implications}\label{subsec:interpretation}

Theorem~\ref{thm:sharp-obstruction} shows that, in dimensions $n\ge3$, subharmonicity alone does not yield a H\"{o}rmander-type weighted existence theorem for the Dirac operator. The estimate \eqref{eq:naive-estimate} fails already for the minimal-norm solution corresponding to the explicit data $f_m$, so the obstruction is not a matter of improving the constant in the bound.

Accordingly, the stronger condition \eqref{eq:structure-kappa} appearing in Theorem~\ref{thm:classical-L2} should be interpreted as a coercivity hypothesis within the classical Bochner framework. The counterexample also motivates the weighted identity developed in Section~\ref{section:6}, which yields solvability results for weights outside that classical regime, including the Gaussian weight.

It remains of interest to determine whether analogous phenomena persist for broader classes of anisotropic or non-polynomial weights and for other Clifford-valued systems.

\section{A weighted identity and coercive estimates}\label{section:6}

The obstruction obtained in Section~\ref{section:5} shows that the classical weighted \(L^{2}\) identity from \cite{JiZhu17} is not sufficient once the condition
\[
\Delta\varphi \geq \kappa|\nabla\varphi|^{2}, \quad \kappa > \frac{n-2}{n},
\]
fails. In this section we derive an alternative weighted identity for the conjugated unknown \(U:=ue^{-\varphi/2}\). The identity will be used to obtain coercive estimates for radial powers, for the weight \(x_{1}^{2}\), and for small anisotropic perturbations of the Gaussian weight.

\begin{proposition}[General weighted identity]
\label{prop:general-identity}
Let \(\Omega \subset \mathbb{R}^{n}\) be a domain, and let \(\varphi, \eta \in C^{\infty}(\Omega, \mathbb{R})\).
Define
\[
\YY = \sum_{j=1}^{n} Y_{j} e_{j},
\qquad Y_{j} \in C^{\infty}(\Omega,\mathbb{R}),
\]
and write \(U := u e^{-\varphi/2}\) for \(u \in C_c^{\infty}(\Omega,\mathbb{R}_{n})\).
Then
\[
\begin{aligned}
&\|D_{\varphi}^{*}u\|_{\varphi}^{2}
  - \left\| \eta U + \sum_{j=1}^{n} e_{j} \YY \partial_{j} U \right\|_{0}^{2} \\
&= \int_{\Omega}
  \left(
   \frac{1}{4}|D\varphi|^{2}
   - \eta^{2}
   - 2 \sum_{j=1}^{n} \partial_{j}(\eta Y_{j})
   + \frac{1}{2}\Delta |\YY|^{2}
  \right) |u|^{2} e^{-\varphi}\,dV \\
&\quad
  - \langle DU, (D\varphi + 2\eta \YY) U \rangle_{0}
  + \sum_{j=1}^{n} \int_{\Omega} (1 - |\YY|^{2}) |\partial_{j} U|^{2}\,dV \\
&\quad
  + \sum_{j,k=1}^{n}
   \left\langle
    \partial_{k}\!\left( \YY e_{k} e_{j} \YY \right)\partial_{j} U,
    U
   \right\rangle_{0}.
\end{aligned}
\]
\end{proposition}

\begin{proof}
Since \(U = u e^{-\varphi/2}\), Proposition~\ref{prop:formal-adjoint} gives
\[
D_{\varphi}^{*}u
  = e^{\varphi/2}\left(DU-\frac12 D\varphi\cdot U\right).
\]
Hence
\begin{equation}\label{eq:Dstar_norm}
\begin{aligned}
\|D_\varphi^*u\|_\varphi^2
 &= \left\|DU-\frac12 D\varphi\cdot U\right\|_0^2 \\
 &= \|DU\|_0^2
  +\frac14\|D\varphi\cdot U\|_0^2
  -\langle DU,D\varphi\cdot U\rangle_0 \\
 &= -\langle \Delta U,U\rangle_0
  +\frac14\|D\varphi\cdot U\|_0^2
  -\langle DU,D\varphi\cdot U\rangle_0.
\end{aligned}
\end{equation}

Set
\[
Q:=\eta U+\sum_{j=1}^n e_j\YY\partial_j U.
\]
Expanding \(\|Q\|_0^2\) gives
\[
\|Q\|_0^2
 = \|\eta U\|_0^2
  + \Sigma_1
  - \Sigma_2
  + 2\langle \eta\YY U,DU\rangle_0,
\]
where
\[
\Sigma_1:=\left\|\sum_{j=1}^n e_j\YY\partial_j U\right\|_0^2
\]
and
\[
\Sigma_2:=4\sum_{j=1}^n\langle \eta U,Y_j\partial_j U\rangle_0.
\]
Indeed, using \(e_j\YY=-\YY e_j-2Y_j\) and Lemma~\ref{lem:weighted-inner-basic}(2) with
\(\varphi=0\), we get
\[
2\sum_{j=1}^n\langle \eta U,e_j\YY\partial_jU\rangle_0
 = -4\sum_{j=1}^n\langle \eta U,Y_j\partial_j U\rangle_0
  +2\langle \eta\YY U,DU\rangle_0.
\]

We first compute \(\Sigma_2\). Since \(\eta\) and the \(Y_j\) are real-valued,
\[
\Sigma_2
 = 4\sum_{j=1}^n \sum_A \int_\Omega \eta Y_j U_A\,\partial_j U_A\,dV
 = 2\sum_{j=1}^n \sum_A \int_\Omega \eta Y_j\,\partial_j(U_A^2)\,dV.
\]
Because \(U\) has compact support, integration by parts yields
\begin{equation}\label{eq:Sigma2}
 \Sigma_2
  = -2\sum_{j=1}^n \int_\Omega \partial_j(\eta Y_j)\,|U|^2\,dV.
\end{equation}

Next, for each $j,k$ we have
\begin{align*}
\langle e_j\YY\partial_j U, e_k\YY\partial_k U\rangle_0
 &= -\langle e_k e_j\YY\partial_j U, \YY\partial_k U\rangle_0 \\
 &= \langle \YY e_k e_j\YY\partial_j U, \partial_k U\rangle_0,
\end{align*}
where we used Lemma~\ref{lem:weighted-inner-basic}(2) to move $e_k$ and
Lemma~\ref{lem:weighted-inner-basic}(3) to move the vector field $\YY$, noting that
$\overline{\YY}=-\YY$. Since $U$ has compact support, integration by parts in the
$x_k$-variable yields
\begin{align}
 \Sigma_1
  &= -\sum_{j,k=1}^n
   \Bigl(
     \left\langle \partial_k(\YY e_k e_j\YY)\partial_j U,U\right\rangle_0
     + \left\langle \YY e_k e_j\YY\,\partial_k\partial_j U,U\right\rangle_0
   \Bigr). \label{eq:Sigma1}
\end{align}
We now simplify the second-order term. If $j\neq k$, then
\[
\YY e_k e_j\YY + \YY e_j e_k\YY
= \YY (e_k e_j + e_j e_k) \YY = 0,
\]
while $\partial_k\partial_j U = \partial_j\partial_k U$. Hence the off-diagonal terms cancel in
pairs. For $j=k$, we have
\[
\YY e_j e_j\YY = -\YY^2 = |\YY|^2,
\]
because $\YY$ is a vector field and therefore $\YY^2=-|\YY|^2$. Consequently,
\[
-\sum_{j,k=1}^n
 \left\langle \YY e_k e_j\YY\,\partial_k\partial_j U,U\right\rangle_0
 = -\langle |\YY|^2\Delta U,U\rangle_0.
\]
Finally, integration by parts gives
\begin{equation}\label{eq:final step}
\langle (|\YY|^2-1)\Delta U,U\rangle_0
 = \int_\Omega
  \left(
    (1-|\YY|^2)\sum_{j=1}^n |\partial_j U|^2
    +\frac12 \Delta|\YY|^2\,|U|^2
  \right)\,dV.
\end{equation}

Combining \eqref{eq:Dstar_norm}, \eqref{eq:Sigma2}, \eqref{eq:Sigma1}, and \eqref{eq:final step},
and recalling that \(|U|^2=|u|^2e^{-\varphi}\), we obtain the stated identity.
\end{proof}

\begin{remark}
In dimension \(n=2\), if \(\varphi\in C^{2}(\Omega,\mathbb R)\) is subharmonic and
\(|D\varphi|>0\), one may choose
\[
\YY=\frac{D\varphi}{|D\varphi|},
\qquad
\eta=-\frac12 |D\varphi|.
\]
Then \(|\YY|\equiv1\), the coefficient \(\frac14|D\varphi|^{2}-\eta^{2}\) vanishes,
\(-2\sum_{j=1}^{2}\partial_j(\eta Y_j)=\Delta\varphi\), and
\(D\varphi+2\eta\YY=0\). The term involving \(1-|\YY|^{2}\) also vanishes, and the last
summation in Proposition~\ref{prop:general-identity} disappears by the two-dimensional
Clifford algebra identities. The twisted norm becomes
\[
2\int_\Omega \left(\sum_{j,A}|\partial_j u_A|^2-\frac12 |Du|^2\right)e^{-\varphi}\,dV.
\]
Hence Proposition~\ref{prop:general-identity} reduces to the two-dimensional weighted
identity used in Proposition~\ref{prop:apriori}(1), and therefore recovers the estimate in
Theorem~\ref{thm:classical-L2}(1).
\end{remark}

We now apply Proposition~\ref{prop:general-identity} with specific choices of \(\eta\) and \(\YY\). The first applications concern radial and quadratic weights.

\begin{proposition}[Radial weight estimate]\label{prop:radial-estimate}
Let $m \in \mathbb{R}\setminus\{0\}$, let $\varphi = |x|^m$, and let
$\Omega \subset \mathbb{R}^n \setminus \{0\}$ be a domain. For any
$u \in C_c^{\infty}(\Omega, \mathbb{R}_n)$, we have
\[
\|D_{\varphi}^* u\|_{\varphi}^2
 = \left\| |x| D \left( \frac{x}{|x|^2} u \right) \right\|_{\varphi}^2
  + m^2 \int_{\Omega} |x|^{m-2} |u|^2 e^{-\varphi}\,dV.
\]
In particular, for the Gaussian weight $\varphi = |x|^2$,
\[
\|D_{\varphi}^* u\|_{\varphi}^2 \ge 4 \|u\|_{\varphi}^2.
\]
\end{proposition}

\begin{proof}
Write \(r:=|x|\) and \(U:=u e^{-\varphi/2}\). We apply
Proposition~\ref{prop:general-identity} with
\[
 \YY = \frac{x}{r},
 \qquad
 \eta = -\frac12 m r^{m-1}+\frac{2-n}{r}.
\]
Then \(|\YY|=1\) and \(\Delta |\YY|^{2}=0\). Since
\[
 D\varphi = m r^{m-2}x,
 \qquad
 D\varphi + 2\eta\YY = 2(2-n)\frac{x}{r^2},
\]
the general identity becomes
\begin{align}
&\|D_{\varphi}^* u\|_{\varphi}^2
 - \left\| \eta U+\sum_{j=1}^n e_j\YY\partial_j U \right\|_{0}^2 \nonumber\\
&=
\int_\Omega
 \left(
   \frac14 |D\varphi|^2 - \eta^2 - 2\sum_{j=1}^n\partial_j(\eta Y_j)
 \right)
 |u|^2e^{-\varphi}\,dV
 +2(n-2)\left\langle DU,\frac{x}{r^2}U\right\rangle_0 \nonumber\\
&\quad
 +\sum_{j,k=1}^n
  \left\langle \partial_k(\YY e_k e_j\YY)\partial_j U,U\right\rangle_0.
\label{eq:radial-step1}
\end{align}
Now
\[
 \eta Y_j = -\frac12 \partial_j\varphi + (2-n)\frac{x_j}{r^2},
\]
so
\[
 -2\sum_{j=1}^n \partial_j(\eta Y_j)
 = \Delta\varphi - 2(2-n)\sum_{j=1}^n \partial_j\!\left(\frac{x_j}{r^2}\right).
\]
A direct computation gives
\[
 \frac14 |D\varphi|^2 - \eta^2
  = m(2-n)r^{m-2} - \frac{(n-2)^2}{r^2},
\]
and since
\[
 \Delta\varphi = m(m+n-2)r^{m-2},
 \qquad
 \sum_{j=1}^n \partial_j\!\left(\frac{x_j}{r^2}\right)=\frac{n-2}{r^2},
\]
the coefficient of \(|u|^2e^{-\varphi}\) in \eqref{eq:radial-step1} simplifies to
\[
 m^2 r^{m-2}+\frac{(n-2)^2}{r^2}.
\]

We next identify the twisted norm. Since
\[
D\!\left(\frac{x}{r^2}\right)=\frac{2-n}{r^2},
\]
we have
\begin{align*}
 r D\!\left(\frac{x}{r^2}u\right)
  &= r D\!\left(\frac{x}{r^2}U e^{\varphi/2}\right) \\
  &= e^{\varphi/2}
   \left(
     \frac{2-n}{r}U
     + \sum_{j=1}^n e_j\frac{x}{r}\partial_j U
     + \frac12 \sum_{j=1}^n e_j\frac{x}{r}U\,\partial_j\varphi
   \right).
\end{align*}
Because $D\varphi = m r^{m-2}x$ and $x^2=-r^2$, we compute
\[
 \frac12\sum_{j=1}^n e_j\frac{x}{r}U\,\partial_j\varphi
  = \frac12\frac{D\varphi\,x}{r}U
  = \frac12\frac{m r^{m-2}x^2}{r}U
  = -\frac12 m r^{m-1}U.
\]
Combining these identities, we obtain
\[
 r D\!\left(\frac{x}{r^2}u\right)
 = e^{\varphi/2}
  \left(
    \eta U+\sum_{j=1}^n e_j\YY\partial_j U
  \right).
\]
Therefore
\begin{equation}\label{eq:radial-step2}
 \left\| \eta U+\sum_{j=1}^n e_j\YY\partial_j U \right\|_{0}
 = \left\| r D\!\left(\frac{x}{r^2}u\right)\right\|_{\varphi}.
\end{equation}

Finally,
\[
 \sum_{j,k=1}^n
 \left\langle \partial_k(\YY e_k e_j\YY)\partial_j U,U\right\rangle_0
 =(n-2)\sum_{j=1}^n
  \left\langle
    \frac{x e_j-e_j x}{r^2}\partial_j U,U
  \right\rangle_0,
\]
and another integration by parts gives
\begin{equation}\label{eq:radial-step3}
 \sum_{j,k=1}^n
 \left\langle \partial_k(\YY e_k e_j\YY)\partial_j U,U\right\rangle_0
 = -2(n-2)\left\langle DU,\frac{x}{r^2}U\right\rangle_0
  - (n-2)^2\int_\Omega \frac{|u|^2}{r^2}e^{-\varphi}\,dV.
\end{equation}
Substituting \eqref{eq:radial-step2} and \eqref{eq:radial-step3} into \eqref{eq:radial-step1}
yields the desired identity.
\end{proof}

\begin{proposition}[Single quadratic weight estimate]
\label{prop:single-quadratic}
Let \(\varphi = x_{1}^{2}\) and let \(\Omega\subset\mathbb{R}^{n}\) be a domain.
For any \(u \in C_c^{\infty}(\Omega, \mathbb{R}_{n})\), one has
\[
\|D_{\varphi}^{*}u\|_{\varphi}^{2}
 = \| D(e_{1} u) \|_{\varphi}^{2} + 2 \|u\|_{\varphi}^{2}
 \ge 2 \|u\|_{\varphi}^{2}.
\]
\end{proposition}

\begin{proof}
Write \(U:=u e^{-\varphi/2}\). Since \(D\varphi=2x_{1}e_{1}\), we apply
Proposition~\ref{prop:general-identity} with
\[
 \YY=e_{1},
 \qquad
 \eta=-x_{1}.
\]
Then \(|\YY|=1\), \(\Delta|\YY|^{2}=0\), and
\[
 D\varphi+2\eta\YY = 2x_{1}e_{1}-2x_{1}e_{1}=0.
\]
Moreover,
\[
 \frac14 |D\varphi|^{2}-\eta^{2}=x_{1}^{2}-x_{1}^{2}=0,
 \qquad
 -2\partial_{1}(\eta Y_{1}) = -2\partial_{1}(-x_{1})=2,
\]
while \(\partial_k(\YY e_k e_j \YY)=0\) because \(\YY=e_{1}\) is constant. Hence
Proposition~\ref{prop:general-identity} gives
\begin{equation}\label{eq:single-quadratic-step1}
 \|D_{\varphi}^{*}u\|_{\varphi}^{2}
 - \left\| -x_{1}U+\sum_{j=1}^{n} e_{j}e_{1}\partial_{j}U \right\|_{0}^{2}
 = 2\|u\|_{\varphi}^{2}.
\end{equation}

It remains to identify the twisted norm. Since \(u=Ue^{x_{1}^{2}/2}\),
\[
 D(e_{1}u)
 = \sum_{j=1}^{n} e_{j}\partial_{j}(e_{1}Ue^{x_{1}^{2}/2})
 = e^{x_{1}^{2}/2}
  \left(
   \sum_{j=1}^{n} e_{j}e_{1}\partial_{j}U
   + \frac12\sum_{j=1}^{n} e_{j}e_{1}U\,\partial_{j}\varphi
  \right).
\]
Because \(\partial_{j}\varphi=0\) for \(j\neq1\) and \(e_{1}e_{1}=-1\),
\[
 \frac12\sum_{j=1}^{n} e_{j}e_{1}U\,\partial_{j}\varphi
  = x_{1}e_{1}e_{1}U
  = -x_{1}U.
\]
Therefore
\[
 D(e_{1}u)
  = e^{\varphi/2}
   \left(
    -x_{1}U+\sum_{j=1}^{n} e_{j}e_{1}\partial_{j}U
   \right),
\]
and thus
\[
 \left\| -x_{1}U+\sum_{j=1}^{n} e_{j}e_{1}\partial_{j}U \right\|_{0}
 = \|D(e_{1}u)\|_{\varphi}.
\]
Substituting this into \eqref{eq:single-quadratic-step1} proves the claim.
\end{proof}

We next consider small anisotropic perturbations of the Gaussian weight.

\begin{proposition}[Perturbed Gaussian weights estimate]\label{prop:perturbed-gaussian}
Let $\varphi=\sum_{i=1}^{n}a_{i}x_{i}^{2}$ with $|a_{i}-1|<\varepsilon$ for $\varepsilon>0$
sufficiently small, and let
$\Omega\subset\mathbb{R}^{n}\setminus\overline{B(0,1)}$ be a domain. Then there exists
$\varepsilon_{0}=\varepsilon_{0}(n)>0$ such that, whenever \(0<\varepsilon<\varepsilon_{0}\),
\[
\|D_{\varphi}^{*}u\|_{\varphi}^{2} \ge 3\|u\|_{\varphi}^{2}
\qquad
\text{for all }u\in C_c^{\infty}(\Omega,\mathbb{R}_{n}).
\]
\end{proposition}

\begin{proof}
Set
\[
D\varphi = 2\sum_{i=1}^{n} a_i x_i e_i,
\qquad
|D\varphi|^2 = 4\sum_{i=1}^{n} a_i^2 x_i^2,
\qquad
\Delta\varphi = 2\sum_{i=1}^{n} a_i.
\]
Write $U:=u e^{-\varphi/2}$ and choose
\[
\YY=\frac{D\varphi}{|D\varphi|},
\qquad
\eta=-\frac12 |D\varphi|+\frac{2(2-n)}{|D\varphi|}.
\]
Since $|\YY|=1$, Proposition~\ref{prop:general-identity} gives
\begin{align}
&\|D_\varphi^*u\|_\varphi^2
 - \left\| \eta U+\sum_{j=1}^n e_j\YY\partial_j U \right\|_{0}^2 \nonumber\\
&=
\int_\Omega \left(\Delta\varphi+2(2-n)-\frac{4(2-n)^2}{|D\varphi|^2}
   -\sum_{j=1}^n\partial_j\left(\frac{8a_j(2-n)x_j}{|D\varphi|^2}\right)\right)
   |u|^2 e^{-\varphi}\,dV \nonumber\\
&\quad
 - \left\langle DU,\frac{4(2-n)D\varphi}{|D\varphi|^2}U \right\rangle_{0}
 + \sum_{j,k=1}^n
  \left\langle \partial_k\left(\frac{D\varphi\, e_k e_j\,D\varphi}{|D\varphi|^2}\right)
  \partial_j U,U \right\rangle_{0}.
\label{eq:perturbed-main0}
\end{align}
Indeed,
\[
 \eta Y_j = -a_jx_j+\frac{4(2-n)a_jx_j}{|D\varphi|^2},
\]
so
\[
 -2\sum_{j=1}^n\partial_j(\eta Y_j)
 = \Delta\varphi - \sum_{j=1}^n\partial_j\left(\frac{8(2-n)a_jx_j}{|D\varphi|^2}\right),
\]
and
\[
 \frac14|D\varphi|^2-\eta^2
 = 2(2-n)-\frac{4(2-n)^2}{|D\varphi|^2}.
\]

We now add and subtract the radial-Gaussian coefficients associated with $Y=x/|x|$.
Equation~\eqref{eq:radial-step3}, proved earlier for this choice of $Y$, is an identity in the
function $U$ alone and therefore applies here as well:
\begin{equation}\label{eq:radial-step3-perturbed}
 \sum_{j,k=1}^n
 \left\langle \partial_k\left(\frac{x\,e_ke_j\,x}{|x|^2}\right)\partial_j U,U\right\rangle_0
 = -2(n-2)\left\langle DU,\frac{x}{|x|^2}U\right\rangle_0
  - (n-2)^2\int_\Omega \frac{|u|^2}{|x|^2}e^{-\varphi}\,dV.
\end{equation}
Substituting this cancellation into \eqref{eq:perturbed-main0}, we may rewrite the identity as
\begin{equation}\label{eq:perturbed-main}
 \|D_\varphi^*u\|_\varphi^2
 - \left\| \eta U+\sum_{j=1}^n e_j\YY\partial_j U \right\|_{0}^2
 = 4\|u\|_\varphi^2 + A_1 + A_2 + A_3,
\end{equation}
where
\begin{align*}
A_1&=\int_\Omega E_\varepsilon(x)\,|u|^2 e^{-\varphi}\,dV,\\
E_\varepsilon(x)&:=2\sum_{i=1}^{n}(a_i-1)
  -\frac{4(2-n)^2}{|D\varphi|^2}
  -\sum_{j=1}^{n}\partial_j\left(\frac{8a_j(2-n)x_j}{|D\varphi|^2}\right)
  -\frac{(n-2)^2}{|x|^2},\\
A_2&=\left\langle DU,
   2(2-n)\left(\frac{x}{|x|^2}-\frac{\sum_i a_ix_i e_i}{\sum_j a_j^2x_j^2}\right)U
  \right\rangle_0,\\
A_3&=\sum_{j,k=1}^n
   \left\langle
    \partial_k\left(
     \frac{D\varphi\,e_ke_j\,D\varphi}{|D\varphi|^2}
     -\frac{x\,e_ke_j\,x}{|x|^2}
    \right)\partial_j U,
    U
   \right\rangle_0.
\end{align*}

For $a_i\equiv1$ one has $E_\varepsilon\equiv0$. More generally, if $|a_i-1|<\varepsilon$ and
$|x|\ge1$, then $|D\varphi|^2=4|x|^2+O_n(\varepsilon |x|^2)$, hence
\[
 \frac{1}{|D\varphi|^2}=\frac{1}{4|x|^2}+O_n\!\left(\frac{\varepsilon}{|x|^2}\right),
 \qquad
 \partial_j\left(\frac{a_jx_j}{|D\varphi|^2}\right)
  = \frac14\,\partial_j\left(\frac{x_j}{|x|^2}\right)+O_n\!\left(\frac{\varepsilon}{|x|^2}\right).
\]
It follows that $E_\varepsilon(x)=O_n(\varepsilon)$ uniformly on $\Omega$, so there exists a constant
$C_n>0$, depending only on $n$, such that
\begin{equation}\label{eq:A1-bound}
 |A_1|\le C_n\varepsilon \, \|u\|_\varphi^2.
\end{equation}
Likewise, the coefficient differences in $A_2$ and $A_3$ are $O_n(\varepsilon |x|^{-1})$, and therefore
\begin{equation}\label{eq:A23-rough}
 |A_2+A_3|
 \le C_n\varepsilon
   \left(
    \|u\|_\varphi^2
    + \sum_{j=1}^n\int_\Omega \frac{|\partial_j U|^2}{|D\varphi|^2}\,dV
   \right).
\end{equation}

To control the last term we set $\psi:=2\log|D\varphi|$. Then $e^{-\psi}=|D\varphi|^{-2}$, and
\eqref{eq:basic-Bochner} applied with the weight $\psi$ yields
\begin{align}
 \sum_{j=1}^n\int_\Omega\frac{|\partial_j U|^2}{|D\varphi|^2}\,dV
 &= \frac12\left(\|\delta_\psi U\|_\psi^2+\|DU\|_\psi^2
  -\int_\Omega \Delta\psi\,|U|^2e^{-\psi}\,dV\right) \nonumber\\
 &= \frac12\left(\|DU-D\psi\cdot U\|_\psi^2+\|DU\|_\psi^2
  -\int_\Omega \Delta\psi\,|U|^2e^{-\psi}\,dV\right). \label{eq:psi-bochner}
\end{align}
Now $|D\varphi|^2 \simeq |x|^2$ uniformly for $\varepsilon$ small, while
$|D\psi|\lesssim_n |x|^{-1}$ and $|\Delta\psi|\lesssim_n |x|^{-2}$. Since $|x|\ge1$,
these imply
\[
 \|D\psi\cdot U\|_\psi^2
 + \left|\int_\Omega \Delta\psi\,|U|^2e^{-\psi}\,dV\right|
 \le C_n \|U\|_0^2.
\]
Using $\|a-b\|^2\le 2\|a\|^2+2\|b\|^2$ in \eqref{eq:psi-bochner}, we obtain
\begin{equation}\label{eq:psi-bound}
 \sum_{j=1}^n\int_\Omega\frac{|\partial_j U|^2}{|D\varphi|^2}\,dV
 \le \frac32\int_\Omega \frac{|DU|^2}{|D\varphi|^2}\,dV + C_n\|U\|_0^2.
\end{equation}

It remains to estimate $\int |DU|^2/|D\varphi|^2$. From
\[
 \|D_\varphi^*u\|_\varphi^2+\|u\|_\varphi^2
 = \left\|DU-\frac12 D\varphi\cdot U\right\|_0^2+\|U\|_0^2
\]
and completion of the square, we obtain
\begin{equation}\label{eq:DU-over-Dphi}
 \|D_\varphi^*u\|_\varphi^2+\|u\|_\varphi^2
 \ge \int_\Omega \frac{4|DU|^2}{4+|D\varphi|^2}\,dV.
\end{equation}
If $\varepsilon$ is sufficiently small, then $|D\varphi|^2\ge 2$ on $\Omega$, and therefore
\[
 \frac{4}{4+|D\varphi|^2}\ge \frac{1}{|D\varphi|^2}.
\]
Combining this with \eqref{eq:DU-over-Dphi} gives
\[
 \int_\Omega \frac{|DU|^2}{|D\varphi|^2}\,dV
 \le \|D_\varphi^*u\|_\varphi^2+\|u\|_\varphi^2.
\]
Substituting into \eqref{eq:psi-bound} and then into \eqref{eq:A23-rough}, we obtain
\begin{equation}\label{eq:A23-bound}
 |A_2+A_3|
 \le \frac32 C_n\varepsilon\,\|D_\varphi^*u\|_\varphi^2
   + \frac52 C_n\varepsilon\,\|u\|_\varphi^2.
\end{equation}

Finally, \eqref{eq:perturbed-main}, \eqref{eq:A1-bound}, and \eqref{eq:A23-bound} imply
\[
 \left(1-\frac32 C_n\varepsilon\right)\|D_\varphi^*u\|_\varphi^2
 \ge \left(4-\frac72 C_n\varepsilon\right)\|u\|_\varphi^2.
\]
Choosing $\varepsilon_0>0$ so that
\[
 1-\frac32 C_n\varepsilon_0 >0,
 \qquad
 \frac{4-\frac72 C_n\varepsilon_0}{1-\frac32 C_n\varepsilon_0}\ge 3,
\]
we conclude that $\|D_\varphi^*u\|_\varphi^2 \ge 3\|u\|_\varphi^2$.
\end{proof}

The same identity also yields the following abstract consequence involving \(\Delta\varphi\).

\begin{proposition}\label{prop:general-application}
Let \(\Omega \subset \mathbb{R}^{n}\) be a domain and let
\(\varphi \in C^{2}(\Omega, \mathbb{R})\) be subharmonic with \(|D\varphi| > 0\).
For \(u \in C_c^{\infty}(\Omega, \mathbb{R}_{n})\), write \(U:=u e^{-\varphi/2}\).
Then
\begin{align*}
\|D_{\varphi}^{*}u\|_{\varphi}^{2}
&- \left\| \sum_{j=1}^{n} e_{j} \frac{D\varphi}{|D\varphi|} \partial_{j} u \right\|_{\varphi}^{2} \\
&= \int_{\Omega} \Delta\varphi |u|^{2} e^{-\varphi}\,dV
  + \sum_{j,k=1}^{n}
   \left\langle
     \partial_{k}\left( \frac{D\varphi\, e_{k} e_{j}\, D\varphi}{|D\varphi|^{2}} \right)
     \partial_{j} U, U
   \right\rangle_{0}.
\end{align*}
In particular, if there exist constants \(k>0\) and \(\varepsilon\in(0,1)\) such that
\begin{align*}
k \|D_{\varphi}^{*}u\|_{\varphi}^{2}
&+ (1 - \varepsilon) \int_{\Omega} \Delta\varphi |u|^{2} e^{-\varphi}\,dV \\
&+ \sum_{j,k=1}^{n}
   \left\langle
     \partial_{k}\left( \frac{D\varphi\, e_{k} e_{j}\, D\varphi}{|D\varphi|^{2}} \right)
     \partial_{j} U, U
   \right\rangle_{0}
 \ge 0.
\end{align*}
for all such \(u\), then
\[
\|D_{\varphi}^{*}u\|_{\varphi}^{2}
 \ge \frac{\varepsilon}{1+k}
  \int_{\Omega} \Delta\varphi |u|^{2} e^{-\varphi}\,dV.
\]
\end{proposition}

\begin{proof}
We apply Proposition~\ref{prop:general-identity} with
\[
 \YY = \frac{D\varphi}{|D\varphi|},
 \qquad
 \eta = -\frac12 |D\varphi|.
\]
Then \(|\YY|=1\), \(\frac14|D\varphi|^2-\eta^2=0\), and
\[
 -2\sum_{j=1}^n \partial_j(\eta Y_j)
 = -2\sum_{j=1}^n \partial_j\!\left(-\frac12\,\partial_j\varphi\right)
 = \Delta\varphi,
\]
while \(D\varphi+2\eta\YY=0\). Therefore Proposition~\ref{prop:general-identity} yields
\begin{align*}
\|D_{\varphi}^{*}u\|_{\varphi}^{2}
&- \left\| \eta U + \sum_{j=1}^{n} e_{j}\YY \partial_{j} U \right\|_{0}^{2} \\
&= \int_{\Omega} \Delta\varphi |u|^{2} e^{-\varphi}\,dV
  + \sum_{j,k=1}^{n}
   \left\langle
     \partial_{k}\left( \frac{D\varphi\, e_{k} e_{j}\, D\varphi}{|D\varphi|^{2}} \right)
     \partial_{j} U, U
   \right\rangle_{0}.
\end{align*}
Finally,
\[
 \sum_{j=1}^n e_j\frac{D\varphi}{|D\varphi|}\partial_j u
  = e^{\varphi/2}
   \left(
    -\frac12 |D\varphi|\,U
    + \sum_{j=1}^n e_j\frac{D\varphi}{|D\varphi|}\partial_j U
   \right),
\]
because
\[
 \frac12\sum_{j=1}^n e_j\frac{D\varphi}{|D\varphi|}U\,\partial_j\varphi
  = \frac12\frac{(D\varphi)^2}{|D\varphi|}\,U
  = -\frac12 |D\varphi|\,U.
\]
Hence
\[
 \left\| \eta U+\sum_{j=1}^{n} e_{j}\YY \partial_{j} U \right\|_{0}
 = \left\| \sum_{j=1}^{n} e_{j} \frac{D\varphi}{|D\varphi|} \partial_{j} u \right\|_{\varphi},
\]
which proves the identity.

For the implication, move the last summation term to the left-hand side and combine the
assumed inequality with the displayed identity. Since the twisted norm is nonnegative, we obtain
\[
 (1+k)\|D_\varphi^*u\|_\varphi^2
 \ge \varepsilon \int_\Omega \Delta\varphi |u|^2 e^{-\varphi}\,dV,
\]
which is the stated estimate.
\end{proof}

The weight \(\varphi=x_1^2\) is \emph{not} covered directly by Proposition~\ref{prop:general-application}, because \(|D\varphi|=2|x_1|\) vanishes on the hyperplane \(\{x_1=0\}\). That case is handled separately in Proposition~\ref{prop:single-quadratic}. Likewise, the Gaussian weight \(\varphi=|x|^2\) on all of \(\mathbb R^n\) is treated via Proposition~\ref{prop:radial-estimate} on punctured domains, since the choice \(\YY=x/|x|\) is singular at the origin.

The estimates obtained in this section will be combined with Lemma~\ref{lem:L2 existence} in Section~\ref{section:7} to derive the corresponding existence theorems.

\section{Weighted \(L^2\) existence theorems for the Dirac equation}\label{section:7}

In this section we combine the coercive estimates of Section~\ref{section:6} with Lemma~\ref{lem:L2 existence} to obtain the corresponding weighted existence theorems for the Dirac equation.

\begin{theorem}[Weighted \(L^2\) existence theorems for the Dirac equation]
\label{thm:L2-existence}
The following existence results hold.
\begin{enumerate}
\item \textbf{Radial weights.} Let $m \in \mathbb{R}\setminus\{0\}$, let $\varphi = |x|^m$,
and let $\Omega \subset \mathbb{R}^{n} \setminus \{0\}$ be a domain.
For every $f \in L_{\varphi}^{2}(\Omega, \mathbb{R}_{n})$ satisfying
\[
\int_{\Omega} \frac{|f|^{2}}{|x|^{m-2}} e^{-\varphi}\,dV < \infty,
\]
there exists a solution $u \in L_{\varphi}^{2}(\Omega, \mathbb{R}_{n})$ to \(Du=f\) such that
\[
\|u\|_{\varphi}^{2}
 \le \int_{\Omega} \frac{|f|^{2}}{m^{2}|x|^{m-2}} e^{-\varphi}\,dV.
\]
In particular, for the Gaussian weight \(\varphi=|x|^2\), every
\(f \in L_{\varphi}^{2}(\Omega,\mathbb{R}_{n})\) admits a solution \(u\) with
\[
\|u\|_{\varphi}^{2} \le \frac14 \|f\|_{\varphi}^{2}.
\]
If \(\Omega=\mathbb{R}^{n}\setminus\{0\}\), then the constant \(1/4\) is sharp.

\item \textbf{Single quadratic weight.} Let \(\varphi=x_1^2\) and let
\(\Omega\subset\mathbb{R}^{n}\) be a domain. For every
\(f \in L_{\varphi}^{2}(\Omega, \mathbb{R}_{n})\), there exists a solution
\(u \in L_{\varphi}^{2}(\Omega, \mathbb{R}_{n})\) to \(Du=f\) such that
\[
\|u\|_{\varphi}^{2} \le \frac12 \|f\|_{\varphi}^{2}.
\]

\item \textbf{Perturbed Gaussian weights.} Let
\(\varphi = \sum_{i=1}^{n} a_i x_i^2\) with \(|a_i-1|<\varepsilon\) for some sufficiently small
\(\varepsilon>0\), and let \(\Omega\) be a domain exterior to \(B(0,1)\).
For every \(f \in L_{\varphi}^{2}(\Omega, \mathbb{R}_{n})\), there exists a solution \(u\) to
\(Du=f\) such that
\[
\|u\|_{\varphi}^{2} \le \frac13 \|f\|_{\varphi}^{2}.
\]
\end{enumerate}
\end{theorem}

\begin{proof}
Each statement follows from Lemma~\ref{lem:L2 existence} once the corresponding coercive estimate
for \(D_\varphi^*\) has been established.

\begin{enumerate}
\item Proposition~\ref{prop:radial-estimate} yields
\[
\|D_{\varphi}^{*}u\|_{\varphi}^{2}
 \ge m^2 \int_{\Omega}|x|^{m-2}|u|^2e^{-\varphi}\,dV
 \qquad
 \text{for all }u\in C_c^\infty(\Omega,\mathbb{R}_n),
\]
and Lemma~\ref{lem:L2 existence} gives the existence statement and estimate.

To prove sharpness in the Gaussian case on \(\Omega=\mathbb{R}^{n}\setminus\{0\}\),
argue by contradiction. Assume that there exists a constant $C<1/4$ such that for every
$f\in L_\varphi^2(\Omega,\mathbb R_n)$ there exists $g\in L_\varphi^2(\Omega,\mathbb R_n)$ with
$Dg=f$ and
\[
 \|g\|_\varphi^2 \le C\|f\|_\varphi^2.
\]
Let $w\in C_c^\infty(\Omega,\mathbb R_n)$. For each $f$, choose such a solution $g_f$. Since
$Dg_f=f$ in the sense of distributions and $w$ has compact support, we have
\[
 \langle w,f\rangle_\varphi
 = \langle we^{-\varphi},f\rangle_0
 = \langle we^{-\varphi},Dg_f\rangle_0
 = \langle D(we^{-\varphi}),g_f\rangle_0
 = \langle e^{\varphi}D(we^{-\varphi}),g_f\rangle_\varphi
 = \langle D_\varphi^*w,g_f\rangle_\varphi.
\]
Therefore,
\[
 |\langle w,f\rangle_\varphi|
 \le \|D_\varphi^*w\|_\varphi\,\|g_f\|_\varphi
 \le \sqrt{C}\,\|D_\varphi^*w\|_\varphi\,\|f\|_\varphi.
\]
Taking the supremum over all $f\neq0$ yields
\[
 \|w\|_\varphi^2 \le C\|D_\varphi^*w\|_\varphi^2
 \qquad \text{for all } w\in C_c^\infty(\Omega,\mathbb R_n).
\]

Now set
\[
 x := \sum_{j=1}^n x_j e_j,
 \qquad
 u_0 := 2|x|^2-n.
\]
Choose $\chi_m\in C_c^\infty(\Omega)$ such that $0\le \chi_m\le 1$,
$\chi_m\equiv1$ on $\{2/m\le |x|\le m\}$,
$\operatorname{supp}\chi_m\subset\{1/m\le |x|\le 2m\}$, and
$|x|\,|\nabla\chi_m|\le C_0$ on $\Omega$, where $C_0$ is independent of $m$.
Set $w_m:=\chi_m x$. Then $w_m\in C_c^\infty(\Omega,\mathbb R_n)$,
$w_m\to x$ in $L_\varphi^2$, and, for $\varphi=|x|^2$,
\[
 D_\varphi^* w_m
 = D(\chi_m x) - 2x\,\chi_m x
 = \chi_m(2|x|^2-n) + (D\chi_m)x
 = \chi_m u_0 + (D\chi_m)x.
\]
Because $D\chi_m$ is a vector field, Lemma~\ref{lem:Clifford-norm} gives
\[
 |(D\chi_m)x| = |D\chi_m|\,|x| = |\nabla\chi_m|\,|x| \le C_0.
\]
Moreover, $(D\chi_m)x$ is supported in
$\{ |x|<2/m\}\cup\{ |x|>m\}$, so dominated convergence implies
$(D\chi_m)x\to0$ in $L_\varphi^2$. Hence $D_\varphi^*w_m\to u_0$ in $L_\varphi^2$.
Passing to the limit in the previous coercive inequality, we obtain
\[
 \|u_0\|_\varphi^2 \ge C^{-1}\|x\|_\varphi^2.
\]
On the other hand,
\[
 \|x\|_\varphi^2
 = \sigma_{n-1}\int_0^\infty r^{n+1}e^{-r^2}\,dr
 = \frac12\sigma_{n-1}\Gamma\left(\frac n2+1\right),
\]
while
\begin{align*}
\|u_0\|_\varphi^2
 &= \sigma_{n-1}\int_0^\infty (2r^2-n)^2 r^{n-1}e^{-r^2}\,dr \\
 &= 2\sigma_{n-1}\Gamma\left(\frac n2+1\right)
 = 4\|x\|_\varphi^2.
\end{align*}
Therefore $4\|x\|_\varphi^2 = \|u_0\|_\varphi^2 \ge C^{-1}\|x\|_\varphi^2$, so $C\ge1/4$, which is a
contradiction. Thus the constant $1/4$ is sharp.

\item Proposition~\ref{prop:single-quadratic} gives
\[
\|D_{\varphi}^{*}u\|_{\varphi}^{2}
 \ge 2 \|u\|_{\varphi}^{2},
 \qquad u\in C_c^\infty(\Omega,\mathbb{R}_n),
\]
and Lemma~\ref{lem:L2 existence} yields the stated conclusion.

\item Proposition~\ref{prop:perturbed-gaussian} gives
\[
\|D_{\varphi}^{*}u\|_{\varphi}^{2}
 \ge 3 \|u\|_{\varphi}^{2},
 \qquad u\in C_c^\infty(\Omega,\mathbb{R}_n),
\]
and Lemma~\ref{lem:L2 existence} again applies.
\end{enumerate}
This completes the proof.
\end{proof}

By the factorization $\Delta=-D^{2}$, we obtain the following weighted solvability result for the Poisson equation.

\begin{corollary}[$L^2$-Existence for the Laplace Equation] \label{cor:laplace-existence}
Let $\varphi = |x|^2$ and $\Omega \subset \mathbb{R}^{n} \setminus \{0\}$ be a domain. For every $f \in L_{\varphi}^{2}(\Omega, \mathbb{R}_{n})$, there exists a solution $u \in L_{\varphi}^{2}(\Omega, \mathbb{R}_{n})$ to the Poisson equation
\[
\Delta u = f
\]
satisfying the estimate
\[
\|u\|_{\varphi}^{2} \leq \frac{1}{16} \|f\|_{\varphi}^{2}.
\]
Furthermore, if $f$ is real-valued, then $u$ can be chosen to be real-valued with the same estimate.
\end{corollary}

\begin{proof}
By Theorem \ref{thm:L2-existence}(1), there exists $v \in L_{\varphi}^{2}(\Omega, \mathbb{R}_{n})$ such that
\[
Dv = f \quad \text{and} \quad \|v\|_{\varphi}^{2} \leq \frac{1}{4} \|f\|_{\varphi}^{2}.
\]
Applying Theorem \ref{thm:L2-existence}(1) again to the data $-v$, there exists $w \in L_{\varphi}^{2}(\Omega, \mathbb{R}_{n})$ such that
\[
Dw = -v \quad \text{and} \quad \|w\|_{\varphi}^{2} \leq \frac{1}{4} \|v\|_{\varphi}^{2} \leq \frac{1}{16} \|f\|_{\varphi}^{2}.
\]
Define $u := w$. Then,
\[
\Delta u = -D^2 w = D(-D w) = Dv = f,
\]
and $\|u\|_{\varphi}^{2} = \|w\|_{\varphi}^{2} \leq \frac{1}{16} \|f\|_{\varphi}^{2}$.

If $f$ is real-valued, then the real part of $u$, $\operatorname{Re}(u)$, is also a solution to $\Delta(\operatorname{Re} u) = f$ and satisfies $\|\operatorname{Re} u\|_{\varphi}^{2} \leq \|u\|_{\varphi}^{2} \leq \frac{1}{16} \|f\|_{\varphi}^{2}$.
\end{proof}

We conclude the section with an abstract $L^2$ existence theorem for the Dirac operator.

\begin{theorem}\label{thm:general-L2-existence}
Let $\Omega \subset \mathbb{R}^{n}$ be a domain and let
$\varphi \in C^2(\Omega, \mathbb{R})$ be subharmonic with $|D\varphi|>0$ on $\Omega$.
Assume that there exist constants \(k>0\) and \(\varepsilon\in(0,1)\) such that
\[
k \|D_{\varphi}^{*}u\|_{\varphi}^{2}
 + (1 - \varepsilon) \int_{\Omega} \Delta \varphi |u|^2 e^{-\varphi}\,dV
 + \sum_{j,k=1}^{n}
  \left\langle
   \partial_k \left( \frac{D\varphi \, e_k e_j D\varphi}{|D\varphi|^2} \right)
   \partial_j U, U
  \right\rangle_0
 \ge 0
\]
for all \(u \in C_c^\infty(\Omega,\mathbb{R}_n)\), where \(U:=u e^{-\varphi/2}\).
Then, for every \(f \in L_{\varphi}^{2}(\Omega,\mathbb{R}_{n})\) satisfying
\[
\int_{\Omega} \frac{|f|^2}{ \Delta \varphi } e^{-\varphi}\,dV < \infty,
\]
there exists a solution \(u \in L_{\varphi}^{2}(\Omega,\mathbb{R}_{n})\) to \(Du = f\) such that
\[
\|u\|_{\varphi}^{2}
 \le \frac{1 + k}{\varepsilon}
  \int_{\Omega} \frac{|f|^2}{\Delta \varphi} e^{-\varphi}\,dV.
\]
\end{theorem}

\begin{proof}
By Proposition~\ref{prop:general-application},
\[
\|D_{\varphi}^{*}u\|_{\varphi}^{2}
 \ge \frac{\varepsilon}{1 + k}
  \int_{\Omega} \Delta \varphi\, |u|^2 e^{-\varphi}\,dV
\qquad
\text{for all }u\in C_c^\infty(\Omega,\mathbb{R}_n).
\]
The conclusion now follows directly from Lemma~\ref{lem:L2 existence}.
\end{proof}

\section{Concluding remarks}\label{section:8}

This paper studies weighted $L^{2}$ estimates for the Euclidean Dirac operator in higher dimensions. Theorem~\ref{thm:1.2} shows that the two-dimensional subharmonic-weight principle does not extend to dimensions $n\ge3$: on $\mathbb R^{n}\setminus\overline{B(0,1)}$ with $\varphi=n\log|x|$, the Laplacian of the weight does not by itself control the minimal $L^{2}_{\varphi}$ solution. Accordingly, within the classical Bochner framework, conditions such as \eqref{eq:structural-kappa} should be viewed as genuine coercivity assumptions.

The positive results are obtained from the weighted identity in Proposition~\ref{prop:general-identity}, formulated for the conjugated unknown $U:=ue^{-\varphi/2}$ and suitable auxiliary multipliers. This identity yields weighted solvability for radial powers $|x|^{m}$, for the quadratic weight $x_{1}^{2}$, and for small anisotropic perturbations of the Gaussian weight on exterior domains. In the Gaussian case the constant $1/4$ in \eqref{eq:sharp-estimate-gauss-weight} is sharp. Through the factorization $\Delta=-D^{2}$, the same estimates also give weighted solvability for the Poisson equation.

Several questions remain open, including extensions to Dirac-type operators on manifolds and to more general anisotropic or non-polynomial weights.

\section*{Declarations}
\textbf{Funding} This work was supported by the National Natural Science Foundation of China (Grant No.~12571090).

\textbf{Competing interests} The authors declare that they have no competing interests.

\end{document}